%% file: m.tex
\documentclass[reqno,12pt]{amsart}

\include{environ}
\include{symbols}

\begin{document}

\title[AFEM Techniques for the PBE]
      {Adaptive Finite Element Modeling Techniques \\
       for the Poisson-Boltzmann Equation}

\author[M. Holst]{Michael Holst}
\email{mholst@math.ucsd.edu}

\author[J.A. McCammon]{James Andrew McCammon}
\email{jmccammon@ucsd.edu}

\author[Z. Yu]{Zeyun Yu}
\email{yu@math.ucsd.edu}

\author[Y.C. Zhou]{Youngcheng Zhou}
\email{zhou@math.ucsd.edu}

\author[Y. Zhu]{Yunrong Zhu}
\email{zhu@math.ucsd.edu}

\address{Department of Mathematics\\
         University of California San Diego\\ 
         La Jolla CA 92093}
\address{Department of Physics\\
         University of California San Diego\\ 
         La Jolla CA 92093}
\address{Department of Chemistry \& Biochemistry\\
         University of California San Diego\\ 
         La Jolla CA 92093}
\address{Center for Theoretical Biological Physics (CTBP)\\
         University of California San Diego\\ 
         La Jolla CA 92093}
\address{National Biomedical Computational Resource (NBCR)\\
         University of California San Diego\\ 
         La Jolla CA 92093}
\address{Howard Hughes Medical Institute (HHMI)\\
         University of California San Diego\\ 
         La Jolla CA 92093}

\thanks{MH was supported in part by 
        NSF Awards~0715146, 0821816, 0915220, and 0822283 (CTBP),
        NIH Award P41RR08605-16 (NBCR), DOD/DTRA Award HDTRA-09-1-0036,
        CTBP, NBCR, NSF, and NIH}
\thanks{JAM was supported in part by 
        NIH, NSF, HHMI, CTBP, and NBCR}
\thanks{ZY, YZ, and YZ were supported in part by 
        NSF Award~0715146, CTBP, NBCR, and HHMI}

\date{\today}


\input{abs}

\maketitle

\clearpage

%\vspace*{-0.5cm}
{\footnotesize
\tableofcontents
}
\vspace*{-0.5cm}

\input{intro}

\input{pbe}

\input{afem}

\input{contract}

\input{meshgen}
\input{numeg}

\input{conc}

\input{ack}

\bibliographystyle{abbrv}
\bibliography{../bib/coupling,../bib/pnp,../bib/meshgen,../bib/mjh,../bib/library,../bib/books}

%\clearpage
%\input{app}

%\vspace*{0.5cm}

\end{document}

%% file: environ.tex
\usepackage{color} % black,white,red,green,blue,cyan,magenta,yellow
\usepackage{ifpdf}
\ifpdf
    \usepackage[pdftex]{graphicx}
    \usepackage[pdftex]{hyperref}
    \hypersetup{
        unicode=false,          % non-Latin characters in Acrobat’s bookmarks
        pdftoolbar=true,        % show Acrobat’s toolbar?
        pdfmenubar=true,        % show Acrobat’s menu?
        pdffitwindow=false,     % window fit to page when opened
        pdfstartview={FitH},    % fits the width of the page to the window
        pdftitle={MCP Article},      % title
        pdfauthor={Michael Holst},   % author
        pdfsubject={Mathematics},    % subject of the document
        pdfcreator={Michael Holst},  % creator of the document
        pdfproducer={Michael Holst}, % producer of the document
        pdfkeywords={PDE, analysis, mathematical physics}, % list of keywords
        pdfnewwindow=true,      % links in new window
        colorlinks=true,        % false: boxed links; true: colored links
        linkcolor=red,          % color of internal links
        citecolor=blue,         % color of links to bibliography
        filecolor=magenta,      % color of file links
        urlcolor=cyan           % color of external links
    }

\else
    \usepackage{graphicx}
    \usepackage{pstricks}
    
    \newcommand{\href}[2]{#2}
\fi

\usepackage{times}
\usepackage{amsfonts}

\usepackage{amsmath}
\usepackage{amsthm}
\usepackage{amssymb}
\usepackage{amsbsy}
\usepackage{amscd}

\usepackage{enumerate}
\usepackage{verbatim}
\usepackage{subfigure}

\newtheorem{theorem}{Theorem}[section]

\newtheorem{lemma}[theorem]{Lemma}

\newtheorem{example}[theorem]{Example}

\numberwithin{equation}{section}  %amsmath command: tie counter to section

  \newcounter{mnote}
  \let\oldmarginpar\marginpar
    \renewcommand\marginpar[1]{\-\oldmarginpar[\raggedleft\footnotesize #1]%
    {\raggedright\footnotesize #1}}
\newenvironment{enumerateX}
{\begin{list}{\arabic{enumi})}{
\usecounter{enumi}
\leftmargin 2.5em\topsep 0.5em\itemsep -0.0em\labelwidth 50.0em}}
{\end{list}}

\newenvironment{assumptionX}
{\begin{trivlist}\em\item[\hskip\labelsep{\bf Assumption:}]}{\end{trivlist}}

\definecolor{myblue}{rgb}{0.2,0.2,0.7}
\definecolor{mygreen}{rgb}{0,0.6,0}
\definecolor{mycyan}{rgb}{0,0.6,0.6}
\definecolor{myred}{rgb}{0.9,0.2,0.2}
\definecolor{mymagenta}{rgb}{0.9,0.2,0.9}
\definecolor{mywhite}{rgb}{1.0,1.0,1.0}
\definecolor{myblack}{rgb}{0.0,0.0,0.0}

%% file: symbols.tex
\newcommand{\beq}{\begin{equation}}
\newcommand{\eeq}{\end{equation}}
\newcommand{\beqa}{\begin{eqnarray}}
\newcommand{\eeqa}{\end{eqnarray}}
\newcommand{\tbar}{|\hspace*{-0.15em}|\hspace*{-0.15em}|}

\newcommand{\leqs}{\leqslant}      % Nice >= 
     
\newcommand{\geqs}{\geqslant}      % Nice =<
      
\newcommand{\R}{{\mathbb R}}       % Real numbers.
\newcommand{\cM}{{\mathcal M}}

\newcommand{\cP}{{\mathcal P}}

\newcommand{\cS}{{\mathcal S}}
\newcommand{\cT}{{\mathcal T}}

\newcommand{\cX}{{\mathcal X}}

\newcommand{\bD}{{\bf D}}

%% file: abs.tex
\begin{abstract}
We consider the design of an effective and reliable adaptive finite element
method (AFEM) for the nonlinear Poisson-Boltzmann equation (PBE).
We first examine the two-term regularization technique for the continuous 
problem recently proposed by Chen, Holst, and Xu based on the removal of the 
singular electrostatic potential inside biomolecules; this technique made 
possible the development of the first complete solution and approximation 
theory for the Poisson-Boltzmann equation, the first provably convergent 
discretization, and also allowed for the development of a provably
convergent AFEM.
However, in practical implementation, this two-term regularization exhibits
numerical instability.
Therefore, we examine a variation of this regularization technique which 
can be shown to be less susceptible to such instability.
We establish {\em a priori} estimates and other basic results for the 
continuous regularized problem, as well as for Galerkin finite element
approximations.
We show that the new approach produces regularized continuous and discrete
problems with the same mathematical advantages of the original regularization.
We then design an AFEM scheme for the new regularized problem, 
and show that the resulting AFEM scheme is accurate and reliable, by proving
a contraction result for the error.
This result, which is one of the first results of this type for nonlinear 
elliptic problems, is based on using continuous and discrete
{\em a priori} $L^{\infty}$ estimates to establish quasi-orthogonality.
To provide a high-quality geometric model as input to the AFEM algorithm,
we also describe a class of feature-preserving adaptive mesh generation 
algorithms designed specifically for constructing meshes of biomolecular 
structures, based on the intrinsic local structure tensor of the molecular 
surface.
All of the algorithms described in the article are implemented
in the Finite Element Toolkit (FETK), developed and maintained 
at UCSD.
The stability advantages of the new regularization scheme are demonstrated 
with FETK through comparisons with the original regularization approach
for a model problem.
The convergence and accuracy of the overall AFEM algorithm is also 
illustrated by numerical approximation of electrostatic solvation energy 
for an insulin protein.
\end{abstract}

%% file: intro.tex
\section{Introduction}
\label{sec:intro}

% These challenging singular features have long been the focus of research 
% activity in computational chemistry and biophysics, and more recently
% have generated growing activity in the numerical analysis and applied 
% mathematics communities.

The Poisson-Boltzmann Equation (PBE) 
has been widely used for modeling the electrostatic interactions 
of charged bodies in dielectric media, such as molecules, ions, and colloids, 
and thus is of importance in many areas of sciences and engineering,
including biochemistry, biophysics, and medicine.
The PBE provides a high fidelity mean-field description of the 
electrostatic interactions and ionic distribution of a solvated biomolecular 
system at the equilibrium state, and entails singularities of different orders 
at the position of the singular permanent charges and dielectric interface. 
The popularity of the PBE model is clearly evidenced by the success of
software packages such as APBS, CHARMM, DelPhi, and UHBD.
We summarize the mathematical PBE model in some detail in 
Section~\ref{sec:pbe}, referring to the classical texts~\cite{Mcqu73,Tanf61} 
for more physical discussions.

While tremendous advances have been made in fast numerical
solution of the PBE over the last twenty 
years (cf.~\cite{Lu.B;Zhou.Y;Holst.M;McCammon.J2008,HBW99,Hols94e}
for surveys of some of this work), 
mathematical results for the PBE (basic understanding of the solution 
theory of the PBE, as well as a basic understanding of approximation 
theory for PBE numerical methods) were fundamentally unsatisfying, 
due to the following questions about the PBE and its numerical solution 
which remained open until 2007:
\begin{enumerateX}
\item Is the PBE well-posed for the dimensionless potential $\tilde{u}$?
\item What function space does the solution $\tilde{u}$ lie in?
\item Can one derive {\em a priori} (energy and/or pointwise) estimates 
      for the solution $\tilde{u}$?
\item Is there an efficient (low-complexity) and reliable (provably convergent
      under uniform mesh refinement)
      numerical method that produces an approximation $u_h$ to the $\tilde{u}$?
\item Is there a provably convergent adaptive method for the PBE?
\end{enumerateX}
That these basic questions were open through 2007 is somewhat remarkable, 
given the popularity of this model.
However, four key features of the PBE model, namely:
(1) the undetermined electrostatic potential at the boundary of a given system;
(2) the singular fixed charge distribution in biomolecules;
(3) the discontinuous dielectric and Debye constants on the irregular 
    dielectric interface (with a possible second interface representing an 
    ion exclusion layer); and 
(4) strong nonlinearity in the case of a strong potential or heavily charged 
    molecules, 
place the PBE into a class of semilinear partial differential equations 
that are fundamentally difficult to analyze, and difficult to 
solve numerically.
In fact, numerical evidence suggested that the most popular algorithms
used for the PBE were actually non-convergent under mesh refinement,
which would put the reliability of scientific results based on 
numerical solution of the PBE in doubt.

To address this issue, in 2007 Chen, Holst, 
and Xu~\cite{CHX06b} used a two-scale decomposition 
as a mathematical technique to answer each of the open questions above
about the PBE, building the first available solution theory and approximation
theory for the PBE.
(A basic existence and uniqueness result using variational arguments
had appeared already in~\cite{Hols94d}.)
A splitting-type treatment of the singular charges was not new,
and is a very natural physical idea first sketched out 
in~\cite{GDLM93} and then also explored numerically in~\cite{ZPVKL96}.
This method decomposes the PBE into 
a Poisson equation with singular charge and uniform dielectric that
determines a singular function $u^s$, and a regularized Poisson-Boltzmann 
equation (RPBE) that determines a smooth correction $u$, with the sum of
the two giving the dimensionless potential: $\tilde{u} = u^s + u$.
This natural splitting technique was exploited 
in~\cite{CHX06b} to show that:
the regularized PBE is well-posed, as is also the full PBE; 
the solution $\tilde{u}$ can be split into a singular function $u^s$ 
(having a simple closed form expression) 
and a smooth remainder $u$ which lies in a well-understood function
space $H^{1+\alpha}$ with $\alpha \geqs 0$; 
the remainder function $u$ is pointwise bounded almost everywhere;
a standard finite element discretization that incorporates the singular
function converges and does so at optimal rate in the limit of uniform
mesh refinement; and finally, an implementable adaptive algorithm
exists that can be proven mathematically to converge to the exact
solution of the PBE.

While this two-scale decomposition made a number of basic mathematical
results possible in~\cite{CHX06b},
the resulting numerical algorithms (both based on uniform mesh refinement
and adaptive mesh refinement) are subject to a hidden instability.
This instability is in fact tied to the two-scale decomposition technique 
itself.
(Example \ref{eg:born_ion_1} in Section \ref{subsect:regforms2} gives
a more complete description of this difficulty.)
While this feature of the splitting has no impact on the mathematical or 
convergence results in~\cite{CHX06b}, the practical impact 
is that algorithms based on this particular decomposition are not accurate 
enough to be competitive with other approaches.
A slightly modified decomposition scheme was proposed by 
Chern et. al~\cite{CLW03} (see also~\cite{GYW07}) and was applied 
together with Cartesian grid-based interface methods to solve the 
PBE for simple structure geometries. 
The new splitting technique gives rise to a modified form of the
regularized PBE with similar structure 
to the splitting scheme in ~\cite{CHX06b}, 
but appears to be more stable.

This article is focused on using a similar decomposition variant to
remove the instability present in the formulation appearing 
in~\cite{CHX06b}, as well as to improve the 
theoretical results and algorithm components of the adaptive 
finite element algorithm described in~\cite{CHX06b}.
In particular, we adopt a variation of the regularization splitting scheme
similar to~\cite{GYW07,CLW03,CHX06b}, involving a 3-term 
expansion rather than a 2-term expansion.
We establish several basic mathematical results for the 3-term
splitting, analogous to those established 
in~\cite{CHX06b} for the original 2-term splitting.
This includes {\em a priori} $L^{\infty}$ estimates,
existence and uniqueness, and discrete estimates for solutions,
and basic error estimates a general class of Galerkin methods.
We then focus specifically on finite element methods, and
design an adaptive finite element method (AFEM)
for solving the resulting regularized PBE.
Due to recent progress in the convergence analysis of 
AFEM for linear and nonlinear 
equations~\cite{Cascon.J;Kreuzer.C;Nochetto.R;Siebert.K2008,HTZ08a}, 
we also substantially improve the AFEM convergence result 
in~\cite{CHX06b} to one which guarantees contraction rather
than just convergence.
We present numerical examples showing the accuracy, efficiency, and stability
of this new scheme.
To provide a high-quality geometric model as input to the AFEM algorithm,
we will also describe a class of feature-preserving adaptive mesh generation 
algorithms designed specifically for constructing meshes of biomolecular 
structures, based on the intrinsic local structure tensor of the molecular 
surface.

While we focus on (adaptive) finite element methods in this article, the 
splitting framework we describe can be incorporated into
finite difference, finite volume, spectral, wavelet, finite element, 
or boundary element methods for the PBE.
While the finite element method has the advantage of exactly representing
the molecular surface (when appropriate mesh generation algorithms
are used; see Section~\ref{sec:mesh}), advances in finite difference 
and finite volume methods include interface discretization 
methods which substantially improve solution accuracy at the dielectric
discontinuity surface~\cite{CLW03,ZZMW06a,ZMW07,GYW07}; 
see also~\cite{Xie.D;Zhou.S2007} for a similar approach using mortar elements.
Boundary element methods for the (primarily linearized) PBE are also
competitive, due to algorithm advances for molecular surface generation 
(see~\cite{HoYu07d,YuHCM08} and Section~\ref{sec:mesh}),
due to emergence of fast multipole codes for surface 
integrals~\cite{Lu.B;Zhang.D;McCammon.J2005,Ben05b,Ben06_pnas}, and due to 
new techniques for nonlinearity~\cite{Boschitsch.A;Fenley.M2004}. 

%%%%%%%%%%%%%%%%%%%%%%%%%%%%%%%%%%%%%%%%%%%%%%%%%%%%
% Outline
%%%%%%%%%%%%%%%%%%%%%%%%%%%%%%%%%%%%%%%%%%%%%%%%%%%%

The remainder of the paper is organized as follows. 
In Section~\ref{sec:pbe}, we give a brief derivation of
the standard form of the PBE,
and then examine the two-scale regularization in~\cite{CHX06b}.
We then describe a second distinct regularization and 
illustrate why it is superior to the original approach 
as a framework for developing numerical methods.
We then quickly assemble the cast of basic mathematical results needed
for the second regularization, which do not immediately follow from the
results established in~\cite{CHX06b} for the original regularization.
In Section~\ref{sec:afem}, we describe an adaptive finite
element method based on residual-type {\em a posteriori} estimates, and
summarize some basic results we need later for the development of a
corresponding convergence theory.
In Section~\ref{sec:conv}, we develop the first AFEM contraction-type result for a 
class of semilinear problems that includes the PBE, substantially improving 
the AFEM convergence result given in~\cite{CHX06b}.
We also include a discussion of our mesh generation toolchain
in Section~\ref{sec:mesh}, which plays a key role in the success of the
overall adaptive numerical method.
Numerical experiments are conducted in Section~\ref{sec:numerical},
where stability of the regularization scheme and convergence 
of the adaptive algorithm are both explicitly demonstrated numerically,
in agreement with the theoretical results established in the paper.
We summarize our results in Section~\ref{sec:summary}.

%% file: pbe.tex
%%%%%%%%%%%%%%%%%%%%%%%%%%%%%%%%%%%%%%%%%%%%%%%%%%%%
\section{The Poisson-Boltzmann Equation (PBE)}
\label{sec:pbe}
%%%%%%%%%%%%%%%%%%%%%%%%%%%%%%%%%%%%%%%%%%%%%%%%%%%%

The PBE can be derived in various ways based on the
statistical description of a system of charged particles in 
electrolytes \cite{Mcqu73,Tanf61}.
A well-known derivation starts with the Poisson equation for the 
electrostatic potential $\phi = \phi(x)$ induced by a charge distribution 
$\rho = \rho(x)$:
\begin{eqnarray*}
-\nabla \cdot (\epsilon \nabla \phi) = \frac{4 \pi}{\epsilon_0}\rho,
\end{eqnarray*}
where $\epsilon=\epsilon(x)$ is a spatially varying dielectric constant 
and $\epsilon_0$ is the dielectric permittivity constant of a vacuum.
The charge distribution $\rho$ may consist of fixed charges $\rho_f$ 
and mobile charges $\rho_m.$ 
The fixed charge distribution $\rho_f$ represents the partially charged 
atoms of the molecules immersed in the aqueous solution; the mobile 
charges $\rho_m$ models the charged ions in the solution. 
With this perspective, the fixed charge distribution $\rho_f$ is 
independent of the potential $\phi.$ The charge distribution $\rho_m$ of mobile ions, however, depends on the 
potential $\phi$ following the Gouy-Chapman or Debye-H\"uckel theories,
and can be modeled by a Boltzmann distribution.
The two charge distributions then take the form
\begin{equation}
\rho_m = \sum_{j=1}^{M} c_j q_j e^{-q_j \phi/kT},
\qquad
\rho_f = \sum_{i=1}^{N} q_i \delta(x_i) 
\quad x_i \in \Omega_m.
\label{eqn:mobile_charge_1}
\end{equation} 
Here for $\rho_{m},$ $M$ is number of ion species, $c_j$ and $q_j$ are the bulk concentration and charge of the $j^{th}$ ion, $k$ is the Boltzmann constant 
and $T$ is the absolute temperature;  and for $\rho_{f},$ there are $N$ charges located at $x_i$ in the molecule region $\Omega_m$ and 
carrying charge $q_i$, where $\delta(x_i)$ is the delta function 
centered at $x_i.$
This gives rise to the 
{\em full} or {\em nonlinear} PBE:
\begin{eqnarray}
-\nabla \cdot (\epsilon \nabla \phi) = \frac{4 \pi}{\epsilon_0} \left ( 
\sum_{i=1}^{N} q_i \delta(x_i) + \sum_{j=1}^{M} c_j q_j e^{-q_j \phi/kT} \right ).
\label{eqn:pbe_full_1}
\end{eqnarray}

A number of variations of the PBE can be derived under
appropriate assumptions. 
For example, for a symmetric 1:1 ionic solution (two ions species with same
but opposite charge) with $M=2,$ bulk concentration $c_j =c$ and charge $q_j = (-1)^{j}q$ for $j=1,2$, 
equation~\eqref{eqn:pbe_full_1} reduces to:
\begin{eqnarray*}
-\nabla \cdot (\epsilon \nabla \phi) + \frac{4 \pi}{\epsilon_0} 2c q ~\sinh \left (  \frac{q \phi}{kT} \right) = \frac{4 \pi}{\epsilon_0} \sum_{i=1}^{N} q_i \delta(x_i).
\label{eqn:pbe_nonlinear_1}
\end{eqnarray*}
We now introduce a dimensionless electrostatic potential 
$\tilde{u} = q \phi/(kT)$, 
and the so-called Debye length 
$l_D = \sqrt{\frac{\epsilon_0 kT}{8\pi c q^2}}$, 
and define the modified Debye-Huckel parameter to be $\kappa=1/l_D$.
After scaling the singular charges we can write the 
final form of the Poisson-Boltzmann equation as:
\begin{equation}
-\nabla \cdot (\epsilon \nabla \tilde{u}) + \kappa^2 \sinh \tilde{u} = f, \label{eqn:pbe_final}
\end{equation}
where $\displaystyle{ f = \sum_{i=1}^{N} z_i \delta(x_i) }$,
with $z_i = 4 \pi q q_i / (\epsilon_0 kT)$.

As analyzed in~\cite{CHX06b}, since the singular function 
$f$ does not belong to $H^{-1}(\Omega)$, equation~\eqref{eqn:pbe_final} does 
not have a solution in $H^1$, or at least the equation does not have
a weak formulation involving the $H^1$ as the test space.
Consequently, standard numerical methods for elliptic equations are not 
guaranteed to produce numerical solutions which converge to the exact
solution to the PBE in the limit of mesh refinement, and numerical evidence 
suggests that in fact standard methods fail to converge.
We now discuss two regularization schemes for the 
PBE which have not only been the basis for the
new solution and approximation theory results for the PBE appearing 
in~\cite{CHX06b}, but also provide a robust framework
for constructing provably convergent numerical algorithms.

%%%%%%%%%%%%%%%%%%%%%%%%%%%%%%%%%%%%%%%%%%%%%%%%%%%%
\subsection{A Natural Regularized Formulation} \label{subsect:regforms1}
%%%%%%%%%%%%%%%%%%%%%%%%%%%%%%%%%%%%%%%%%%%%%%%%%%%%

The first scheme is motivated by the physical interpretation of the solution 
to PBE and decomposes the solution into two components, based on the 
distinct solvent region $\Omega_s$ and molecular region $\Omega_m$
in the model.
This spatial decomposition of the domain $\Omega$, as well as the 
interface $\Gamma$ between $\Omega_s$ and $\Omega_m$, 
is depicted in Figure~\ref{fig:mol}.
\begin{figure}[!ht]
\begin{center}
      \includegraphics[width=0.3\textwidth]{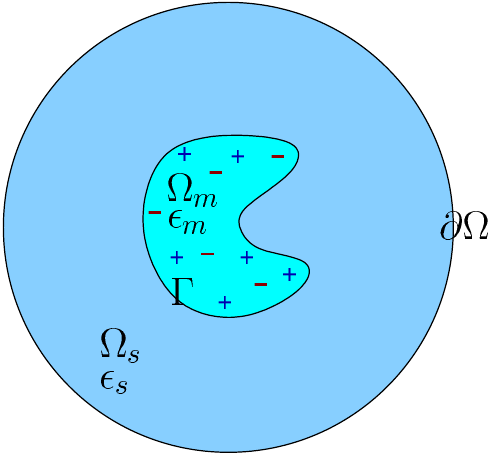}
\end{center}
\caption{Illustration of the solvent region $\Omega_s$, the molecular
         region $\Omega_m$, the interface $\Gamma$, and the two
         distinct dielectric constants $\epsilon_s$ and $\epsilon_m$
         in the two regions.}
\label{fig:mol}
\end{figure}
The component of the solution, which will have singularities but will
be representable in closed-form, is called the self-energy corresponding
to the electrostatic potential.
The second component, which will be much more well-behaved but will not
have a closed-form representation, corresponds to the screening of the 
potential due to high dielectric and mobile ions in the solution region. 
This natural decomposition provides a regularization scheme was proposed 
and explored numerically in \cite{GDLM93,ZPVKL96}. 
In this scheme, the singular component $u^s$ of the electrostatic 
potential is identified as the solution of the following Poisson equation,
the solution of which can be readily assembled from the 
Green's functions (cf.~\cite{StHo2010a}):
\begin{equation}
-\nabla \cdot (\epsilon_m \nabla u^s) = \sum_{i=1}^{N} z_i \delta(x_i),
\qquad
u^s : = \sum_{i=1}^{N} \frac{z_i}{\epsilon_m} \frac{1}{|x-x_i|}.
\label{eqn:us_1}
\end{equation}
Subtracting~\eqref{eqn:us_1} from~\eqref{eqn:pbe_final} 
gives the equation for the regular component $u$:
\begin{eqnarray}
\label{eqn:ur_1}
-\nabla \cdot (\epsilon \nabla u) + \kappa^2 \sinh (u + u^s) = \nabla \cdot ((\epsilon - \epsilon_m) \nabla u^s). 
\end{eqnarray}
Since $\kappa$ vanishes in $\Omega_m$ and $\epsilon - \epsilon_m$ is nonzero 
only in region $\Omega_s$, the right hand side term
$\nabla \cdot ((\epsilon - \epsilon_m) \nabla u^s)$ belongs to 
$H^{-1}$, and a standard $H^1$-weak formulation of \eqref{eqn:ur_1} is 
well-defined.

A variational argument can be used to show existence and uniqueness 
of a weak solution to~\eqref{eqn:us_1} in 
$H^1$ (see~\cite{CHX06b} for this argument, and
also~\cite{Hols94d} for a similar argument in the case
of an alternative regularization).
{\it A priori} $L^{\infty}$ estimates for the solution are
established in~\cite{CHX06b},
which are critical to the development of {\it a priori} error estimates 
for Galerkin (e.g. finite element, wavelet, and spectral) approximations 
of the regular component, and are also critical to the convergence results
for both uniform and adaptive finite element methods developed
in~\cite{CHX06b}.
This two-scale decomposition framework is at the heart of the solution theory, 
approximation theory, and convergence results for adaptive finite element 
methods for the PBE developed in~\cite{CHX06b}.

%%%%%%%%%%%%%%%%%%%%%%%%%%%%%%%%%%%%%%%%%%%%%%%%%%%%
\subsection{An Alternative Regularized Formulation} \label{subsect:regforms2}
%%%%%%%%%%%%%%%%%%%%%%%%%%%%%%%%%%%%%%%%%%%%%%%%%%%%

Since the singular component represents the Coulomb potential in the low 
dielectric environment, it is always much larger than the real potential 
in $\Omega_s,$ where the dielectric constant is high and strong ion screening 
exists.
As a result, the regular component is also much larger in magnitude than the 
full potential in $\Omega_s$, and the decomposition in 
Section \ref{subsect:regforms1} can produce an unstable numerical scheme.
More precisely, relatively small error in the numerical solution of 
regular component could lead to large relative error in the full 
potential, as illustrated in the following example.
\begin{example} \label{eg:born_ion_1}
Let $\Omega_m$ be a unit ball with a unit positive charge at the origin. 
The dielectric constants are $\epsilon_m=2$ and $\epsilon_s=80$ inside and 
outside the ball, respectively. 
Let the modified ionic strength $\kappa = 0.$ 
This so-called Born Ion problem admits an analytical solution for the full 
potential 
$\tilde{u}(r) = \frac{1}{\epsilon_m r} + (\frac{1}{\epsilon_s} - \frac{1}{\epsilon_m})$ in $\Omega_m$ and $\tilde{u}(r) = \frac{1}{\epsilon_s r}$ in $\Omega_s$ where $r = \sqrt{x^2 + y^2 + z^2}$. 
Since the singular component 
$u^s(r) = \frac{1}{\epsilon_m r}$, it follows that the regular component 
$$u(r)=\tilde{u}(r) - u^s(r) =
\frac{1}{\epsilon_s r}(1 - \frac{\epsilon_s}{\epsilon_m}) =-39 \frac{1}{\epsilon_s r} = -39 \tilde{u}(r).$$ 
We assume that the singular component is computed analytically. 
Suppose that the numerical solution of $u(r)$ carries a relative error 
 $e=3\%|u(r)|,$ and assume that this is the only source of numerical error.
This implies the relative error of the final full potential is impacted as
$$ \frac{|e|}{|u(r)|} = 3\%,
\qquad \Longrightarrow \qquad
\frac{|e|}{|\tilde{u}(r)|} = \frac{0.03 |u(r)|}{|\tilde{u}(r)|} 
   = \frac{0.03 \times 39 |\tilde{u}(r)|}{|\tilde{u}(r)|} = 117 \%.
$$
This suggests relative error of 3\% in the numerical 
solution of $u(r)$ will be amplified by 39 times in the relative error 
of the full potential $\tilde{u}(r)$ when $u(r)$ is added to the 
analytical solution $u^s(r).$  
\end{example}

Example \ref{eg:born_ion_1} indicates that unless the regular component is
solved to high accuracy, the full potential could be of low quality, 
and numerical algorithms based on the decomposition may fail.
A second decomposition scheme we now examine 
demonstrates more satisfactory numerical stability. 
Proposed in~\cite{CLW03}, this decomposition splits the potential into 
three parts in the molecular region only.
The first component is the singular 
component $u^s$ defined by~\eqref{eqn:us_1}.
The second component $u^h$ is the harmonic extension of the trace of
the singular component $u^s$ on the molecular surface into the interior 
of the molecule; it is completely determined by the singular 
component $u^s$ and the geometry of 
the molecular surface through the harmonic equation
\begin{equation} 
-\Delta u^h = 0 ~\mbox{in}~\Omega_m,
\qquad
u^h = -u^s ~\mbox{on}~\Gamma,
\label{eqn:uh_1}
\end{equation} 
where $\Gamma$ is the interface between $\Omega_m$ and $\Omega_s.$
Then we set $u^s + u^h =0$ in $\Omega_s.$ 
By definition of $u_h$, this extension is continuous across the interface.
So the complete decomposition reads
\begin{eqnarray}
 \tilde{u} =  u + \frac{\epsilon - \epsilon_s}{\epsilon_m - \epsilon_s} (u^s + u^h) 
   =\begin{cases}
         u^s + u^h + u  & ~\mbox{in}~ \Omega_m\\
 u & ~\mbox{in}~ \Omega_s \\
       \end{cases}. 
\end{eqnarray}
In this decomposition, the regular component $u$ is defined as an interface problem 
\begin{equation} \label{eqn:ur_2}
	\left\{\begin{array}{rlll} 
		-\nabla \cdot (\epsilon \nabla u) + \kappa^2 \sinh u & = &0, &\mbox{in}~\Omega \\
 		\left [ u \right ]_{\Gamma} & = &0, &\mbox{on}~\Gamma, \\
		 \left [\epsilon \frac{\partial u}{\partial n_{\Gamma}} \right ]_{\Gamma} & = & g_{\Gamma}, & \mbox{on}~\Gamma,
\quad
\mbox{with}~ g_{\Gamma} 
:=\epsilon_m \frac{\partial (u^s + u^h)}{\partial n_{\Gamma}}|_{\Gamma},
\\
		 u&=& 0, & |x|\to \infty,
	\end{array}\right.
\end{equation} 
where $n_{\Gamma}$ is the unit out normal of the interface $\Gamma,$ and $[\cdot]$ denotes the jump of enclosed quantity on the given interface as $[v]_{\Gamma} = \lim_{t\to 0} v(x+tn_{\Gamma}) - v(x-tn_{\Gamma}).$
The second interface condition \eqref{eqn:ur_2} arises from
continuity of flux in \eqref{eqn:pbe_final}.
The singular component $u^s$ is given by \eqref{eqn:us_1},
whereas computing $u^h$ is trivial 
using finite element or boundary integral methods.
Therefore, we assume $u^s$ and $u^h$ are known in the following discussion,
and are smooth on $\Gamma.$  
Since the singular component $u^s$ is only applied in the 
interior of the molecule region, a discontinuity appears in
 the remaining component of the potential on the molecular surface. 
The harmonic component $u^h$ is introduced to compensate for this 
discontinuity using harmonic extension,
so that the regular component as defined by equation~\eqref{eqn:ur_2} 
is continuous on the molecular surface.  

Since no decomposition of the potential occurs in $\Omega_s$,
error in numerical solutions of $u$ are not amplified in the full potential.
While mathematically equivalent to the decomposition
in~\cite{CHX06b}, this alternative three term-based
splitting regularization is potentially numerically more favorable than
the original decomposition.
The implementations of this scheme using finite difference interface 
methods \cite{CLW03,GYW07} have proven that it can significantly improve 
the accuracy of the full potential. 
Mirroring the general plan taken
in~\cite{CHX06b}, we will use this alternative
decomposition as the basis for an analysis of the regularized problem,
for the development of an approximation theory, and for the development
of a practical, provably convergent adaptive method.

A final difficultly in solving the regularized form of the
PBE in \eqref{eqn:ur_2} (and other forms of the PBE) is that the 
computational domain is all of space. 
It is standard to truncate space to a bounded 
Lipschitz domain $\Omega$ by posing some artificial (but highly accurate)
boundary condition on $\partial \Omega.$ 
For simplicity, one chooses $\Omega$ to be a ball or cube containing the 
molecule region. 
The solvent region is then defined as $\Omega_s \cap \Omega,$ which will also 
be denoted by $\Omega_s$ without the danger of confusion.
There are various approaches to the choosing boundary condition 
on $\partial\Omega$; using the condition $\tilde{u} = g$ is
standard, where $g$ can be obtained from a known analytical solution to 
some simplification of the 
linearized PBE, and can be chosen to be a smooth function on the boundary.
Far from the molecule region, such analytical solutions provide a highly 
accurate boundary condition approximation for the PBE on the truncated domain.
For other possible constructions of $g,$ 
see \cite{CHX06b,Boschitsch.A;Fenley.M2007,Hols94d}
and the references cited therein. 
Finally, we end up with the regularized PBE (or RPBE)
in a bounded domain $\Omega$,
which becomes the focus for the remainder of the paper:
\begin{equation}
	\label{eqn:reg-pbe}
	\left\{\begin{array}{rlll}
	-\nabla\cdot (\epsilon \nabla u) + \kappa^2 \sinh (u) &=& 0, &\mbox{in } \Omega\\
	
	[u]_{\Gamma} =0 \mbox{ and }{\left[\epsilon \frac{\partial u}{\partial n_{\Gamma}}\right]}_{\Gamma} &=& g_{\Gamma}, &\mbox{on } \Gamma\\

	u|_{\partial \Omega} & = & g, &\mbox{on } \partial\Omega.\\
	\end{array}\right.
\end{equation}
Our main goals for the remainder of the paper are to:
\begin{enumerateX}
\item Establish {\em a priori} $L^{\infty}$ estimates for~\eqref{eqn:reg-pbe},
      leading to a standard argument for well-posedness of the continuous
      and discrete problems.
      Most other mathematical results in the article hinge critically
      on these {\em a priori} estimates.
\item Develop a general approximation theory for \eqref{eqn:reg-pbe}
      by establishing {\em a priori} error estimates for Galerkin methods,
      giving convergence of finite element and other methods;
\item Develop a practical adaptive finite element method for
      \eqref{eqn:reg-pbe} and prove that it is convergent;
\item Develop practical mesh generation algorithms for the domains arising 
      in~\eqref{eqn:reg-pbe} that meet the needs of our finite element methods.
\end{enumerateX}
We note that there are two distinct interface conditions 
in~\eqref{eqn:reg-pbe}, which appears to give it an unusual formulation.
However, the first interface condition $[u]=0$ will be automatically 
satisfied by standard constructions of $C^0$ finite element spaces.
The second interface condition will be embedded into the weak 
form of equation \eqref{eqn:reg-pbe} in a natural way, so that in fact
both interface conditions are quite easily and naturally incorporated 
into finite element (as well as wavelet and spectral) discretizations.
Although fairly complicated schemes arise when considering the regularization
approach with finite difference and finite volume methods,
the interface conditions can be enforced with these discretization
as well
(cf.~\cite{Wang.W2004,Oevermann.M;Klein.R2006,Li.Z;Ito.K2001,Li.Z2003}).

%%%%%%%%%%%%%%%%%%%%%%%%%%%%%%%%%%%%%%%%%%%%%%%%%%%%
\subsection{A priori $L^{\infty}$-Estimates and Well-posedness}
\label{subsect:priori}
%%%%%%%%%%%%%%%%%%%%%%%%%%%%%%%%%%%%%%%%%%%%%%%%%%%%

{\it A priori} $L^{\infty}$ estimates for the solution $u$ to the 
regularized PBE~\eqref{eqn:reg-pbe} are the critical component of
the key mathematical results we need to have in place for the development
of a reliable adaptive method, namely:
(1) well-posedness of the continuous and discrete regularized problems;
(2) {\em a priori} error estimates for Galerkin approximations;
(3) {\em a posteriori} error estimates for Galerkin approximations;
and
(4) auxiliary results for establishing
    convergence (contraction) of AFEM.
The regularized equation~\eqref{eqn:reg-pbe} governing $u$
derived in Section \ref{subsect:regforms2} differs significantly from 
the decomposition used in \cite{CHX06b}, and
as a consequence we now derive the {\em a priori} $L^{\infty}$ estimates.

In what follows, we use standard notation for the $L^p(G)$ spaces, 
$1\leqs p\leqs \infty$, with the norm $\|\cdot\|_{p,G}$ on any subset 
$G\subset \R^{d};$ we use standard notation for Sobolev norms
$\|u\|_{k,p,G} = \|u\|_{W^{k,p}(G)}$ 
where the natural setting here will be $p=2$ and $k=0$ or $k=1.$
 For any functions $v\in L^{p}(G)$ and $w\in L^{q}(G)$ for $p, q\geqs 1$ with $\frac{1}{p} + \frac{1}{q} =1,$ we denote the pairing $(v, w)_{G}$ as 
$(v, w)_{G}:=\int_{G} vw dx.$
If $G=\Omega$ then we also omit it from the norms (or pairing)
to simplify the presentation.

To begin, define an affine subset of $H^1(\Omega)$ as
${H^1_g(\Omega):=\{~v\in H^1(\Omega) ~:~ v=g \mbox{ on } \partial \Omega~\}}$,
and then define 
${\cX_g :=\{~v\in H^1_g(\Omega) ~:~ e^v, e^{-v} \in L^{\infty}(\Omega)~\}}$,
with $\cX_0$ denoting the case when ${g=0}$.
A weak formulation of equation \eqref{eqn:reg-pbe} reads:
Find $u\in \cX_g$ such that
\begin{equation}
	\label{eqn:pbe_weak}
	a(u, v) + ( b(u), v ) = \langle g_{\Gamma},v\rangle_{\Gamma}, 
    \quad \forall v\in H_0^1(\Omega),
\end{equation}
where
$
a(u,v) = (\epsilon \nabla u , \nabla v),$ $(b(u),v ) =   (\kappa^2 \sinh(u),v)$ and $\langle g_{\Gamma},v\rangle_{\Gamma} = \int_{\Gamma} g_{\Gamma} v ds.
$ 
It is easy to verify that the bilinear form in~\eqref{eqn:pbe_weak} 
satisfies:
\begin{equation}
m \|u\|_{1,2}^2 \leqs a(u,u),
\qquad
a(u,v) \leqs M \|u\|_{1,2} \|v\|_{1,2}, 
\qquad \forall u,v \in H_0^{1}(\Omega),
\label{eqn:coercive-bounded}
\end{equation}
where $0 < m \leqs M < \infty$ are constants depending only on the
maximal and minimal values of the dielectric and on the domain. 
The properties~\eqref{eqn:coercive-bounded}
imply the norm on $H_0^{1}(\Omega)$ is equivalent
to the energy norm 
$\tbar\cdot\tbar \colon H_0^{1}(\Omega) \rightarrow \mathbb{R}$,
\begin{equation}
	\tbar u \tbar^2 = a(u,u),
	\qquad
	m \|u\|_{1,2}^2 \leqs \tbar u \tbar^2 \leqs M \|u\|_{1,2}^2.
   \label{eqn:equiv}
\end{equation}
To establish {\em a priori} $L^{\infty}$ estimates,
we further split the solution $u$ to~\eqref{eqn:reg-pbe} into solutions of two sub-problems.
The first sub-problem is a linear elliptic interface problem;
estimates on solutions to this problem are then
utilized in the analyzing the second sub-problem, 
which is a nonlinear elliptic problem without interface conditions.
The second sub-problem is then analyzed using a cut-off function argument 
that exploits a weak formulation of the maximum principle.
More precisely, let $u=u^l + u^n,$ where $u^l\in H_g^1(\Omega)$ satisfies 
the linear elliptic equation
\begin{equation}
	\label{eqn:linear}
	a(u^l, v) = \langle g_{\Gamma}, v\rangle_{\Gamma},
    \quad \forall v\in H_0^1(\Omega),
\end{equation}
and $u^n\in \cX_0$ satisfies the nonlinear elliptic equation:
\begin{equation}
	\label{eqn:nonlinear}
	a(u^n, v) +(b(u^l + u^n) ,v)= 0,
    \quad \forall v\in H_0^1(\Omega),
\end{equation}
where we note that the sum $u = u^l+u^n$ is then the desired solution
to the RPBE~\eqref{eqn:pbe_weak}.
It is easy to see that the linear part $u^l$ is the solution to the 
interface problem:
\begin{equation*}
	\left\{\begin{array}{lll}
		-\nabla \cdot(\epsilon \nabla u^l) &=&  0 \mbox{ in } \Omega\\
		u^l|_{\partial \Omega} = g , \mbox{ and } {\left[\epsilon\frac{\partial u^l}{\partial n_{\Gamma}}\right]}_{\Gamma}&=& g_{\Gamma}; 
	\end{array}\right.
\end{equation*}
while the nonlinear part $u^n$ is the solution to the (homogeneous)
semilinear equation
\begin{equation*}
	\left\{\begin{array}{rlll}
		-\nabla \cdot (\epsilon \nabla u^n) + \kappa^2 \sinh( u^n + u^l) & =& 0 &\mbox{ in } \Omega, \\
u^n & =& 0 &\mbox{ on }  \partial \Omega.
	\end{array}\right.
\end{equation*}
Existence and uniqueness of $u^l$ solving~\eqref{eqn:linear} follows by
standard arguments; furthermore,
if the interface $\Gamma$ to be sufficiently smooth
(e.g.~$\Gamma$ is $C^2$), then $u^{l}\in L^{\infty}(\Omega)$ follows 
immediately from known regularity results for linear interface 
problems (cf.~\cite{CHX06b,Babuska.I1970,Bramble.J;King.J1996,Chen.Z;Zou.J1998,Savare.G1998}).
This makes possible {\em a priori} $L^{\infty}$ estimates
for the nonlinear component, and subsequently the entire regularized solution.
To this end, define
\begin{eqnarray}
\alpha' & = & \arg \max_c \Big( \kappa^2\sinh( 
   c + \sup_{x\in \Omega_s} u_l ) \leqs 0\Big ),
\qquad \alpha^n = \min (\alpha ', 0), 
\label{eqn:alpha} \\
\beta' & = & \arg \min_c \Big( \kappa^2\sinh( 
   c + \inf_{x\in \Omega_s} u_l ) \geqs 0\Big ),
\qquad \beta^n = \max (\beta ', 0).
\label{eqn:beta}
\end{eqnarray}
\begin{lemma}[A Priori $L^{\infty}$ Estimates]
\label{lm:nonlinear}
Suppose that the solution $u^{l}$ to \eqref{eqn:linear} satisfies $u^{l}\in L^{\infty}(\Omega),$ and let $u^n$ be any weak solution of \eqref{eqn:nonlinear}. 
If 
$\alpha^n, \beta^n \in \mathbb{R}$ are as 
defined in~\eqref{eqn:alpha}--\eqref{eqn:beta}, then
\begin{equation}
 \label{eq:un}
 \alpha^n \leqs u^n \leqs \beta^n, ~a.e.~ \mbox{ in } \Omega.
\end{equation}
\end{lemma}
\begin{proof}
The short proof is similar to that
in~\cite{CHX06b,StHo2010a}, which we include for completeness,
due to its critical role in the results throughout the article.
We first define
$$
{\overline{\phi} = (u^n-\beta^n)^+=\max (u^n-\beta^n, 0)},
\quad
{\underline{\phi} = (u^n-\alpha^n)^-=\max (\alpha^n-u^n, 0)}.
$$
Since ${\beta^n \geqs 0}$ and ${\alpha^n \leqs 0}$, 
it follows (cf.~\cite{StHo2010a})
that ${\overline{\phi}, \underline{\phi} \in H_0^1(\Omega)}$,
and can be used as
{\em pointwise non-negative} (almost everywhere) test functions.
For either $\phi=\overline{\phi}$ or 
$\phi=-\underline{\phi}$, we have
$$
(\varepsilon \nabla u^n,\nabla \phi)+(\kappa
^2\sinh(u^n+u^l),\phi) = 0.
$$
Note $\overline{\phi} \geqs 0$ in $\Omega$ and its support set is
$\overline{\mathcal{Y}} = \{x \in \bar \Omega\, | \, u^n(x)\geqs\beta^n\}$. 
On $\overline{\mathcal{Y}}$, we have
$$ \kappa ^2\sinh(u^n+u^l) \geqs \kappa ^2\sinh(\beta ' +
\inf_{x\in \Omega_s} u^l ) \geqs 0.
$$
Similarly, $-\underline{\phi} \leqs 0$ in $\Omega$ with support
$\underline{\mathcal{Y}} = \{x \in \bar \Omega\, | \, u^n(x)\leqs\alpha^n\}$. 
On $\underline{\mathcal{Y}}$, we have
$$ \kappa ^2\sinh(u^n+u^l) \leqs \kappa^2\sinh(\alpha ' +
\sup_{x\in \Omega_s} u^l ) \leqs 0.
$$
Together this implies both
\begin{align*}
& 0 \geqs (\varepsilon \nabla u^n,\nabla \overline{\phi})
    =(\varepsilon \nabla (u^n-\beta^n),\nabla \overline{\phi})
    \geqs (\inf_{x\in \Omega}\varepsilon(x) )\|\nabla \overline{\phi}\|_2^2 \geqs 0,
\\
& 0 \geqs (\varepsilon \nabla u^n,\nabla (-\underline{\phi}))_2
    =(\varepsilon \nabla (\alpha^n-u^n),\nabla \underline{\phi})_2
    \geqs (\inf_{x\in \Omega}\varepsilon(x) ) \|\nabla \underline{\phi}\|_2^2 \geqs 0.
\end{align*}
Using the Poincar\'{e} inequality we have finally
$ 0 \leqs \|\phi\|_{1,2} \lesssim \|\nabla \phi\|_2 \leqs 0,$
giving $\phi =0$,
for either $\phi=\overline{\phi}$ or $\phi=-\underline{\phi}$.
Thus $\alpha^n \leqs u^n\leqs \beta^n$ in $\Omega$.
\end{proof}

We have therefore shown 
that any solution $u \in H^1(\Omega)$ to the regularized 
problem~\eqref{eqn:pbe_weak}
must lie in the set
$$
[\alpha,\beta]_{1,2} := \{~ u \in H^1(\Omega) 
     ~:~ \alpha \leqs u \leqs \beta ~\} 
      \subset H^1(\Omega),
$$
where $\alpha,\beta \in \mathbb{R}$ are 
$$
\alpha = \alpha^n + \inf_{x \in \Omega_s} u_l,
\qquad
\beta = \beta^n + \sup_{x \in \Omega_s} u_l.
$$
Since this ensures $u \in L^{\infty}(\Omega)$, which subsequently
ensures $e^u \in L^2(\Omega)$, we can replace
the set $\cX_g$ with the following
function space as the set to search for solutions to the RPBE:
$$
V = \{~ u \in [\alpha,\beta]_{1,2} 
   ~:~ u = g \mbox{ on } \partial \Omega ~\} \subset \cX_g \subset H^1(\Omega).
$$
Our weak formulation of the RPBE now reads:
\begin{equation}\label{eqn:pbe-weak2}
\mbox{ Find } u \in V \mbox{ such~that }
a(u, v)+(b(u),v)=\langle g_{\Gamma},v\rangle_{\Gamma}, 
\quad \forall v\in H_0^1(\Omega).
\end{equation}
Note that in general $V$ is not a subspace of $H^1(\Omega)$ since it is
not a linear space, due to the inhomogeneous boundary condition requirement.
However, as remarked above, standard results for linear interface problems
imply existence, uniqueness, and {\em a priori} $L^{\infty}$ bounds for $u^l$ 
solving~\eqref{eqn:linear}, leaving only the equation~\eqref{eqn:nonlinear}
for the remainder $u^n$.
Therefore,~\eqref{eqn:pbe-weak2} is mathematically equivalent to
\begin{equation}\label{eqn:pbe-weak3}
\mbox{ Find } u^n \in U \mbox{ such~that }
a(u^n, v)+(b(u^n+u^l),v)=0 \quad \forall v\in H_0^1(\Omega),
\end{equation}
where 
$$
U = \{~ u \in H^1_0(\Omega) ~:~ u_- \leqs u \leqs u_+ ~\} 
    \subset H^1_0(\Omega),
$$
with $u_- = \alpha^n$ and $u_+ = \beta^n$ from Lemma~\ref{lm:nonlinear}.
We now have a formulation~\eqref{eqn:pbe-weak3} that involves looking
for a solution in a well-defined subspace $U$ of
the (ordered) Banach space $X=H^1_0(\Omega)$, and are now prepared to
establish existence (and uniqueness) of the solution.
The argument we use below differs significantly from that used 
in~\cite{CHX06b,Hols94d} for the original regularization.

\begin{theorem}[Existence and Uniqueness of Solutions to RPBE]
\label{thm:pbe-existence}
Let the solution $u^{l}$ to \eqref{eqn:linear} satisfy $u^{l}\in L^{\infty}(\Omega).$ Then there exists a unique weak solution $u^n \in U$ 
to~\eqref{eqn:pbe-weak3}, and
subsequently there exists a unique weak solution 
$u \in V$ to the RPBE~\eqref{eqn:pbe-weak2}.
\end{theorem}

\begin{proof}
We follow the approach in~\cite{CHX06b,Hols94d}.
We begin by defining
${J\colon U \subset H^1_0(\Omega) \to \overline{\mathbb{R}}}$:
$$
J(u)= \int _{\Omega}\frac{\varepsilon}{2}|\nabla u|^2+\kappa
^2\cosh(u+u^l) ~dx.
$$
It is straight-forward to show that
if $u$ is the solution of the optimization problem
\begin{equation}
\label{eqn:rpbe-opt}
J(u)=\inf_{v\in U}J(v) \leqs J(v), ~\forall v \in U,
\end{equation}
then $u$ is the solution of~\eqref{eqn:pbe-weak3}.
We assemble some quick facts about $H^1_0(\Omega)$, 
${U\subset H^1_0(\Omega)}$, and $J$.
\begin{enumerateX}
\item $H^1_0(\Omega)$ is a reflexive Banach space.
\item $U$ is nonempty, convex, and topologically closed as a subset
      of $H^1_0(\Omega)$.
\item $J$ is convex on $U$:
      $J(\lambda u + (1-\lambda)v) \leqs \lambda J(u) + (1-\lambda) J(v)$,
      $\forall u,v \in U$, $\lambda \in (0,1)$.
\end{enumerateX}
By standard results in the calculus of variations (cf.~\cite{StHo2010a}),
we have existence of a solution 
to~\eqref{eqn:rpbe-opt}, and hence to~\eqref{eqn:pbe-weak3}
and~\eqref{eqn:pbe-weak2}, if we can establish two additional
properties of $J$:
\begin{enumerateX}\setcounter{enumi}{3}
\item $J$ is lower semi-continuous on $U$:
       $J(u) \leqs \liminf_{j\to\infty}J(u_j)$, $\forall u_j\to u \in U$.
\item $J$ is coercive on $U$: 
       $J(u) \geqs C_0 \|u\|_{1,2}^2 - C_1$, $\forall u \in U$.
\end{enumerateX}
That $J$ is lower semi-continuous (and in fact, has the stronger
property of weak lower semi-continuity), holds since
$J$ is both convex and Gateaux-differentiable on $U$
(cf.~\cite{StHo2010a} for this and similar results).
That $J$ is coercive follows from $\cosh x \geqs 0$ and the
Poincar\'{e} inequality 
\begin{align*}
J(u) & = \int_{\Omega}\frac{\varepsilon}{2}|\nabla u|^2 + \kappa^2\cosh(u+u^l) dx
\geqs \inf_{x \in \Omega} \frac{\varepsilon(x)}{2} |u|_{1,2}^2
\\
& \geqs \inf_{x \in \Omega} \frac{\varepsilon(x)}{2}
  \left( \frac{1}{2}|u|_{1,2}^2 + \frac{1}{2\rho^2}\|u\|_2^2 \right)
  \geqs C_0 \|u\|_{1,2}^2,
\end{align*}
with $C_0 = ( \inf_{x \in \Omega} \varepsilon(x) )
        \cdot \min \{ 1/4, 1/(4\rho^2) \}$,
where $\rho>0$ is the Poincar\'{e} constant.
It remains to show $u$ is unique.
Assume there are two solutions $u_1$ and $u_2$.
Subtracting~\eqref{eqn:pbe-weak3} for each gives
$
a(u_1-u_2, v)+(b(u_1+u^l)-b(u_2+u^l),v)=0 \quad \forall v\in H_0^1(\Omega).
$
Now take ${v=u_1-u_2}$; monotonicity of the nonlinearity defining
$b$ ensures that ${(b(u_1+u^l)-b(u_2+u^l),u_1-u_2)\geqs 0}$, giving
$
0 \geqs a(u_1-u_2, u_1-u_2) \geqs 2 C_0 \|u_1-u_2\|_{1,2} \geqs 0,
$
where $C_0$ is as above.
This can only hold if $u_1=u_2$.
\end{proof}

In summary, there exists a unique solution 
$u \in V \subset H^{1}(\Omega)$ to the RPBE problem~\eqref{eqn:pbe-weak2},
with compatible barriers $u_-$ and $u_+ \in L^{\infty}$ satisfying
\begin{equation*}
-\infty < u_- \leqs u \leqs u_+ < \infty, \quad \text{a.e.~in}~ \Omega.
\end{equation*}
Moreover, these pointwise bounds combined with a Taylor expansion give
that for any $u,w \in V$ and any $v\in H_0^1(\Omega),$ 
the nonlinearity satisfies a Lipschitz condition:
\begin{equation}
	\label{eqn:K}
( b(u)-b(w),v ) \leqs K \|u-w\|_{2} \|v\|_{2},
\end{equation}
where 
$K = \sup_{\chi\in [u_-, u_+]} \|\kappa^2 \cosh(\chi)\|_{\infty}<\infty$ 
is a constant depending only on the domain, the ionic strength of the 
solvent (embedded in the constant $\kappa$),
and other physical parameters.

%% file: afem.tex
\section{Finite Element Methods (FEM)}
\label{sec:afem}

In this section, we consider (mainly adaptive) finite element methods
for the regularized problem~\eqref{eqn:reg-pbe}. 
For simplicity, we assume $\Omega$ be a bounded polygon domain,
and we triangulate $\Omega$ with a shape regular conforming mesh $\cT_h.$ 
Here $h= h_{\max}$ represents the mesh size which is the maximum diameter 
of elements in $\cT_h.$ 
We further assume that 
\begin{assumptionX}{A1}
the discrete interface $\Gamma_h$ approximates the original 
interface $\Gamma$ to the second order, i.e., $d(\Gamma, \Gamma_h) \leqs ch^2.$ 
\end{assumptionX}
The mesh generator discussed in Section \ref{sec:mesh} provides a practical
tool for generating meshes with this type of approximation quality for the 
interface.
Given such a triangulation $\cT_h,$ we construct the linear finite 
element space 
$$V(\cT_h):=\{v\in H^1(\Omega): v|_{\tau} \in \cP_1(\tau),\;\; \forall \tau \in \cT_h\}.$$
Since we may choose $g$ to be a smooth function on $\partial \Omega,$ 
the Trace Theorem (cf.~\cite{AdFo03,StHo2010a}) ensures there exists 
a fixed function $u_D\in H^1(\Omega)$ such that $u_D =g$ on 
$\partial \Omega$ in the trace sense.
Let $H_D^1(\Omega):=H_0^1(\Omega) + u_D$ be the affine space with the 
specified boundary condition, and $V_D(\cT_h):=V(\cT_h)\cap H_D^1(\Omega)$ 
be the finite element affine space of $H_D^1(\Omega).$ 
In particular, we denote $V_0(\cT_h):=V(\cT_h) \cap H_0^1(\Omega).$ 
For simplicity, we assume the boundary condition $g$ can be represented 
by $u_D$ exactly.  
In practical implementation, we will construct an interpolant of $u_D$ having
sufficient approximation quality such that using the interpolant in place of
$u_D$ will not impact the order of accuracy of the algorithm we build below
for approximating the solution $u$ to the regularized 
problem~\eqref{eqn:reg-pbe}.
A Galerkin finite element approximation of \eqref{eqn:pbe_weak} takes
the form:
Find $u_h \in V_D(\cT_h)$ such that
\begin{equation}
  \label{eqn:disc}
	a(u_h,v) + ( b(u_h), v ) = \langle g_{\Gamma}, v \rangle,
    \quad \forall v \in V_0(\cT_h).
\end{equation}

The primary concerns for this type of approximation technique are the 
following four mathematical questions regarding the 
Galerkin approximation $u_h$:
\begin{enumerateX}
\item Does $u_h$ satisfy discrete {\em a priori} bounds in $L^{\infty}$ and 
      other norms, so that the nonlinearity can be controlled for error 
      analysis?
\item Does $u_h$ satisfy quasi-optimal {\em priori} error estimates,
      so that the finite element method will converge under uniform
      mesh refinement?
\item Can one ensure AFEM (non-uniform mesh refinement) convergence
      $\lim_{k\rightarrow\infty} u_k = u$,
      where $u_k$ is the Galerkin approximation of $u$ at step
      $k$ of AFEM?
\item Can one produce $u_h$ at each step of the (uniform or adaptive)
      refinement algorithm using algorithms which have optimal (linear)
      or nearly optimal space and time complexity?
\end{enumerateX}
The first two questions were answered affirmatively 
in~\cite{CHX06b} for the first regularized formulation.
We give only a brief outline below as to how the arguments for answering 
the first two questions can be modified to establish the analogous results 
for the second regularized formulation here.
The third question was partially answered
in~\cite{CHX06b} for the first regularization,
but we give an improved, more complete answer to this question
below and in Section~\ref{sec:conv}, which is one of the main contributions 
of the paper.
Regarding the fourth question,
due to the discontinuities in the dielectric and the modified Debye-Huckel 
parameter, one must take care to solve the resulting nonlinear algebraic 
systems using robust inexact global Newton methods, combined with modern 
algebraic multilevel-based fast linear solvers in order to produce an 
overall numerical solution algorithm which is reliable and has 
low-complexity.
We do not consider this question further here;
see~\cite{Hols94e,AkHo02,ABH02a,Hols94d} for a complete discussion
in the specific case of the PBE.

\subsection{Discrete $L^{\infty}$ Estimates and Quasi-Optimal {\em A Priori}
            Error Estimates}

For completeness, we quickly answer the first two questions by stating a
result, giving only a very brief outline of how
the result is established for the new regularization, based on modifying
the analogous arguments in~\cite{CHX06b}.
We then focus entirely on the new AFEM contraction results which require
a more complete discussion.
To state the theorem, the following assumption is needed.
\begin{assumptionX}{A2}
For any two adjacent nodes $i$ and $j,$ assume that 
$$a_{ij} = (\epsilon \nabla \phi_i, \nabla \phi_j) \leqs -\frac{c}{h^2} \sum_{e_{ij}\subset \tau} |\tau|, \quad \mbox{ with } \quad c >0,$$
where $e_{ij}$ is the edge associated with these nodes, $\phi_i$ and $\phi_j$ are the basis functions corresponding to nodes $i,j$ respectively, and $|\tau|$ is the volume of $\tau\in \cT.$
\end{assumptionX} 
\begin{theorem}
	\label{thm:disc_apriori}
	If Assumption~A2 holds and $h$ is sufficiently small, the 
    solution to \eqref{eqn:disc} satisfies:
	$$\|u_h\|_{\infty} \leqs C,$$
	where $C$ is independent of $h.$
	Moreover, the quasi-optimal {\em a priori} error estimate holds:
	$$
		\|u-u_h\|_{1,2} \lesssim \inf_{v \in V_D(\cT_h)} \|u-v\|_{1,2}.
	$$
\end{theorem}
\begin{proof}
The proofs of both inequalities are similar to the proofs of
the corresponding results in~\cite{CHX06b}, 
with adjustment to handle the new stabilized splitting.
The second result hinges critically on the Lipschitz property~\eqref{eqn:K},
which in turn relies on the first result together with the continuous
$L^{\infty}$ estimates established in Lemma~\ref{lm:nonlinear}.
\end{proof}

\subsection{Adaptive Finite Element Methods (AFEM)}
\label{subsect:afem}
Adaptive Finite Element Methods (AFEM) build approximation spaces adaptively;
this is done in an effort to use nonlinear approximation so as to meet 
a target quality using spaces having (close to) minimal dimension.
AFEM algorithms are based on an iteration of the form:
$$
\hbox{$\textsf{SOLVE}$} \rightarrow \hbox{$\textsf{ESTIMATE}$} \rightarrow \hbox{$\textsf{MARK}$} \rightarrow \hbox{$\textsf{REFINE}$}
$$
which attempts to equi-distribute error over simplices using subdivision 
driven by {\em a posteriori} error estimates.
Given an initial triangulation $\mathcal T_0$, 
and a parameter $\theta \in (0,1]$, our particular AFEM generates a sequence 
of nested conforming triangulations $\mathcal T_h$, $h>0$,
driven by some local error indicator $\eta(u_h,\tau)$, which
gives rise to a global error indicator $\eta(u_h, \cT_h).$
Schematically, the adaptive algorithm consists of a loop of the following main steps:
\begin{enumerateX}
\item $u_h:=$ \textsf{SOLVE}$\left(\cT_h\right).$
\item $\{\eta(u_h, \tau)\}_{\tau\in \cT_h}:=$ \textsf{ESTIMATE}$\left(u_h, \cT_h\right).$
\item $\cM_h:=$\textsf{ MARK}$\left(\{\eta(u_h, \tau)\}_{\tau\in \cT_h}, \cT_h, \theta\right).$
\item $\cT_*:=$\textsf{REFINE}$\left(\cT_h, \cM_h, \ell\right).$
\end{enumerateX}
In practice, a stopping criteria is placed in Step (2) to terminate the loop.

We will handle each of the four steps as follows:
\begin{enumerateX}
\item \textsf{SOLVE}: We use standard inexact Newton + multilevel
      to produce $U \in V_D({\cT})$ on triangulation $\cT$ (cf.~\cite{Hols2001a}).
      To simplify the analysis here, we assume that the discrete solution
      $U$ is calculated exactly (no round-off error).
      Given a triangulation $\cT$, this defines the procedure:
$$
U := \textsf{SOLVE}(\cT).
$$
\item \textsf{ESTIMATE}: Given a triangulation $\cT$ and a function $U \in V_D({\cT}),$ we compute the elementwise residual error indicator:
$$
\{\eta(U,\tau)\}_{\tau \in \cT} := \textsf{ESTIMATE}(U,\cT).
$$
\item \textsf{MARK}: We use the standard ``D\"orfler marking'': 

Given $\theta \in (0,1]$, we construct a marked subset of elements 
$$
\cM := \textsf{MARK}(\{\eta(U,\tau)\}_{\tau\in \cT}, \cT, \theta) \subseteq \cT,
$$
such that:
\begin{equation}
   \label{eqn:dorfler-property}
   \eta(U,\cM) \geqs \theta \eta(U,\cT).
\end{equation} 
The residual-type error indicator  $\eta(U,\cM)$ over sub-partition $\cM \subseteq \cT$ will be defined precisely in Section~\ref{subsect:aposteriori},

\item \textsf{REFINE}: We use standard non-degenerate bisection-to-conformity 
    methods with known complexity bounds on
    conformity preservation (cf.~\cite{Stevenson.R2005b}).
    In particular, given a triangulation $\cT$, and marked subset $\cM \subseteq \cT$,
    and an integer $\ell \geqs 1$, we produce
$$
\cT_* := \textsf{REFINE}(\cT,\cM, \ell),
$$
    a conforming refinement of $\cT$ with each simplex in $\cM$
    refined at least $\ell$ times.
\end{enumerateX}

%%%%%%%%%%%%%%%%%%%%%%%%%%%%%%%%%%%%%%%%%%%%%%%%%%%%%%%%%%%%%%%%%%%%%%%%%%%%%%
\subsection{Residual {\em A Posteriori} Estimates}
   \label{subsect:aposteriori}
%%%%%%%%%%%%%%%%%%%%%%%%%%%%%%%%%%%%%%%%%%%%%%%%%%%%%%%%%%%%%%%%%%%%%%%%%%%%%%

To make precise the residual-type indicator, we first introduce some standard notation for
the relevant mathematical quantities, and then employ and establish
some of its properties.
\begin{center}
\begin{tabular}{lcl}

$\cT_0$
&=&
Initial conforming simplex triangulation of $\Omega \subset \mathbb{R}^d$. \\

$\cT_h$
&=&
Conforming refinement of $\cT_H$ at the previous step of AFEM. \\

$h_{\tau}$        
 & = &  The diameter of simplex $\tau \in \cT$. \\

$n_F$        
 & = & 
The normal vector to face $F$ of $\tau \in \cT$. \\

$\omega_{\tau}$   
 & = & 
 $~\bigcup~ \{~ \tilde{\tau} \in \cT ~|~
                         \tau \bigcap \tilde{\tau} \ne \emptyset,
                     ~\mathrm{where}~\tau \in \cT ~\}$. \\

$\omega_F$
& = & 
$~\bigcup~ \{~ \tilde{\tau} \in \cT ~|~
                     F \bigcap \tilde{\tau} \ne \emptyset,
                     ~\mathrm{where}~F~\mathrm{is~a~face~of~} \tau \in \cT ~\}$. \\
\end{tabular}
\end{center}
We can now define the following error indicators:
\begin{align}
\eta^2(u_h,{\tau}) &:= h_{\tau}^2 \|b(u_h)\|_{2,\tau}^2
                    + \frac{1}{2} \sum_{F \subset \partial \tau} h_{F} \|n_F \cdot [\epsilon \nabla u_h]\|_{2, F}^2 +
        \sum_{F \subset \partial \tau \cap \Gamma_{h}}h_F \|g_{\Gamma}\|_{2,F}^2,
\label{eqn:res}
\\
\mathrm{osc}^2(u_h,\tau) &:= h_{\tau}^4 \|\nabla u_h\|_{2,\tau}^2 + \sum_{F \subset \partial \tau \cap \Gamma_h}h_F \|g_{\Gamma}-\bar{g}_{\Gamma}\|_{2,F}^2,
\label{eqn:osc}
\end{align}
where $\bar{g}_{\Gamma}$ is the piecewise average on each 
face $F\subset \Gamma_{h}.$ 
For any subset $\cS \subset \cT$, the cumulative indicators are
defined as:
$$
\eta^2(u_h,\cS) := \sum_{\tau \in \cS} \eta^2(u_h,\tau),
\quad \quad
\mathrm{osc}^2(u_h,\cS) := \sum_{\tau \in \cS} \mathrm{osc}^2(u_h,\tau).
$$
From these definitions follows the monotonicity properties:
\begin{eqnarray}
\label{eqn:eta-monotonicity}
\eta(v, \cT_*) & \leqs & \eta(v,\cT), 
    \quad \forall v \in V_D(\cT),
\\
\label{eqn:osc-monotonicity}
\mathrm{osc}(v,\cT_*) & \leqs & \mathrm{osc}(v,\cT),
    \quad \forall v \in V_D(\cT),
\end{eqnarray}
for any refinement $\cT_*$ of $\cT.$
We have then the following global upper bound 
from \cite{Hols2001a,Verfurth.R1994}; the lower-bound 
is also standard and can be found in e.g. \cite{Verfurth.R1994}.
\begin{lemma}[Upper and lower bounds]
   \label{L:estimator-bounds}
Let $u$ and $u_h$ be the solutions to \eqref{eqn:pbe_weak} and \eqref{eqn:disc}, respectively. If the mesh conditions in Assumptions~A1 and~A2 hold, then there exists constants $C_1$ and $C_2$, depending only on $\cT_0$
and the ellipticity constant, such that the following global upper
and lower bounds hold:
\begin{eqnarray}
\label{eqn:upper}
\tbar u-u_h \tbar^2 &\leqs& C_1 \eta^2(u_h,\cT_h),
\\
\label{eqn:lower}
C_2 \eta^2(u_h,\tau) &\leqs& \tbar u-u_h \tbar_{\omega_{\tau}}^2 + \mathrm{osc}^2(u_h,\omega_{\tau}),
\end{eqnarray}
where $ \tbar v \tbar_{\omega_{\tau}}^{2} = \int_{\omega_{\tau}} \epsilon |\nabla v|^{2} dx.$
\end{lemma}
\begin{proof}
Similar to \cite[Theorem 7.1]{CHX06b}, the proof follows the idea of~\cite{Verfurth.R1994} by noticing 
\begin{eqnarray*}
\|b(u_h) - \overline{b}(u_h)\|_{2,\tau} &\leqs& \|\kappa^2(\sinh(u_h) - \sinh(\overline{u}_h))\|_{2,\tau}\\
&\leqs & \sup_{\chi\in [u_-, u_+]} \|\kappa^2\cosh(\chi)\|_{\infty,\tau}\|u_h - \overline{u}_h\|_{2,\tau}\\
&\leqs & C(\kappa, u_-, u_+)h_{\tau} \|\nabla u_h\|_{2,\tau},
\end{eqnarray*}
where we have used Theorem \ref{thm:disc_apriori}. The remaining proof is the same as \cite[Theorem 7.1]{CHX06b}.
\end{proof}
Note for the convergence analysis here, we do not need the lower bound \eqref{eqn:lower}.

%% file: contract.tex
\section{Convergence of AFEM}
\label{sec:conv}

We now develop a convergence analysis of the AFEM iteration by showing 
contraction.
We must establish two additional key auxiliary results first:
an indicator reduction result, and a quasi-orthogonality result,
which generalize two analogous results for the linear case in
\cite{Cascon.J;Kreuzer.C;Nochetto.R;Siebert.K2008}
to a class of nonlinear problems that includes the Poisson-Boltzmann equation.

%%%%%%%%%%%%%%%%%%%%%%%%%%%%%%%%%%%%%%%%%%%%%%%%%%%%%%%%%%%%%%%%%%%%%%%%%%%%%%
\subsection{An Indicator Reduction Lemma}
  \label{sec:indicator-reduction}
%%%%%%%%%%%%%%%%%%%%%%%%%%%%%%%%%%%%%%%%%%%%%%%%%%%%%%%%%%%%%%%%%%%%%%%%%%%%%%

Here we establish a nonlinear generalization of the indicator
reduction result from~\cite[Corollary 4.4]{Cascon.J;Kreuzer.C;Nochetto.R;Siebert.K2008}.
First we prove a 
local perturbation result for the nonlinear
equation (cf.~\cite[Proposition~4.3]{Cascon.J;Kreuzer.C;Nochetto.R;Siebert.K2008}).
We then establish an indicator reduction result.

Let us first introduce a type of nonlinear PDE-specific indicator:
\begin{eqnarray*}
	\eta^2({\bD},\tau) &:=& \|\epsilon\|_{\infty,\omega_{\tau}}^2 + h_{\tau}^2\sup_{\chi\in [u_-, u_+]} \|\kappa^2 \cosh(\chi)\|_{\infty,\tau}^2\\
%	\mathrm{osc}^2({\bD},\tau) &:=& h_{\tau}^2\left( \|\epsilon\|_{\infty}^2 + \sup_{\chi\in [u_-, u_+]} \|\kappa^2 \cosh(\chi)\|_{\infty}^2\right) 
\end{eqnarray*}
For any subset $\cS\subset \cT,$ let $\eta({\bD},\cS):= \max_{\tau\in \cS}\{\eta({\bD},\tau)\}.$ By the definition, it is obvious that $\eta({\bD},\cT)$ is monotone decreasing, i.e.,
\begin{equation}
	\label{eqn:data-monotone}
		\eta({\bD},\cT_*) \leqs \eta({\bD},\cT)
\end{equation}
for any refinement $\cT_*$ of $\cT.$

We now establish (see also~\cite{HTZ08a}) 
a nonlinear generalization of the local
perturbation result appearing as 
in~\cite[Proposition~4.3]{Cascon.J;Kreuzer.C;Nochetto.R;Siebert.K2008}.
This is the key result in generalizing the contraction results in 
\cite[Proposition~4.3]{Cascon.J;Kreuzer.C;Nochetto.R;Siebert.K2008} 
to the semilinear case.
\begin{lemma}[Nonlinear Local Perturbation]
   \label{L:perturbation}
Let $\cT$ be a conforming partition satisfies Assumptions A1 and~A2.
For all $\tau \in \cT$ and for any pair of discrete functions
$v,w \in [u_-,u_+] \cap V_D(\cT),$ it holds that
\begin{eqnarray}
\eta(v,\tau) &\leqs& \eta( w,\tau) 
     + \bar{\Lambda}_1 \eta(\mathbf{D},\tau) \|v-w\|_{1,2,\omega_{\tau}},
%\\
%\mathrm{osc}(v,\tau) &\leqs& \mathrm{osc}(w,{\tau}) 
 %    + \bar{\Lambda}_1 \mathrm{osc}(\mathbf{D},\tau) \|v-w\|_{1,2, \omega_{\tau}},
\end{eqnarray}
where $\bar{\Lambda}_1 > 0$ depends only on the shape-regularity of $\cT_0,$
 and the maximal values that $b$ can
obtain on the $L^{\infty}$-interval $[u_-,u_+].$
\end{lemma}
\begin{proof}
By the definition \eqref{eqn:res} of $\eta,$  we have
\begin{eqnarray*}
	\eta(v,\tau) & \lesssim & \eta(w,\tau) + h_{\tau} \|b(v) - b(w)\|_{2,\tau} + \frac{1}{2}\sum_{F\subset \partial \tau}h_F^{\frac{1}{2}}\|n_F\cdot [\epsilon  \nabla (v-w)]\|_{2,F}
\end{eqnarray*}
Notice that 
$$
	\|b(v) - b(w)\|_{2,\tau}=\|\kappa^2(\sinh(v) -\sinh(w))\|_{2,\tau}\leqs \left(\sup_{\chi\in [u_-, u_+]} \|\kappa^2 \cosh(\chi)\|_{\infty,\tau}\right) \|v-w\|_{2,\tau}.
$$
On the other hand, we also have
$$\|n_F\cdot [\epsilon  \nabla (v-w)]\|_{2,F} \leqs \|\epsilon\|_{\infty,\omega_{\tau}}h_{\tau}^{-\frac{1}{2}}\|\nabla v - \nabla w\|_{2,\omega_{\tau}}.$$
 Therefore ,we get the desired estimate for $\eta.$
\end{proof}

Based on Lemma \ref{L:perturbation}, 
we have the following main estimator reduction
(see also~\cite{HTZ08a}), which generalizes
the linear case appearing in 
\cite[Corollary~4.4]{Cascon.J;Kreuzer.C;Nochetto.R;Siebert.K2008}.
\begin{lemma}[Nonlinear Estimator Reduction]
   \label{L:estimator-reduction}
Let $\cT$ be a partition which satisfies the mesh conditions in
Assumptions~A1 and~A2, and let 
the parameters $\theta \in (0,1]$ and $\ell \geqs 1$ be given. 
Let $\cM=\textsf{MARK}(\{\eta(v,\tau)\}_{\tau\in \cT},\cT,\theta)$,
and let $\cT_* = \textsf{REFINE}(\cT,\cM, \ell).$
If $\Lambda_1 = (d+1)\bar{\Lambda}_1^2/\ell$ with $\bar{\Lambda}_1$
from Lemma~\ref{L:perturbation} and $\lambda = 1-2^{-(\ell/d)} > 0$,
then for all $v \in [u_-,u_+] \cap V_D(\cT),$ 
$v_* \in [u_-,u_+] \cap V_D(\cT_*),$ 
and any $\delta > 0$, it holds that
$$
\eta^2(v_*,\cT_*) 
  \leqs (1 + \delta) [ \eta^2(v,\cT) - \lambda \eta^2(v,\cM) ]
 + (1 + \delta^{-1}) \Lambda_1 \eta^2(\mathbf{D},\cT_0)
                  \tbar v_* - v \tbar^2.
$$
\end{lemma}
\begin{proof}
We follow the proof in \cite[Corollary~4.4]{Cascon.J;Kreuzer.C;Nochetto.R;Siebert.K2008} closely.
We first apply Lemma~\ref{L:perturbation} with $v$ and $v_*$
taken to be in $V_D(\cT_*).$
This gives
$$
\eta(v_*,\tau_*)^2 
\leqs 
(1+\delta) \eta(v,\tau_*) 
+
(1+\delta^{-1}) \bar{\Lambda}_1^2 \eta(\mathbf{D},\tau_*) 
    \|v_*-v\|_{1,2,\omega_{\tau_*}}\quad \forall \tau_*\in \cT_*,
$$
after applying Young's inequality with $\delta > 0.$
We now sum over the elements $\tau_* \in \cT_*,$ using the fact
that for shape regular partitions there is a small finite
number of elements in the overlaps of the patches $\omega_{\tau_*}.$
This gives
$$
\eta(v_*,\cT_*)^2 
\leqs 
(1+\delta) \eta(v,\cT_*) 
+
(1+\delta^{-1}) \Lambda_1^2 \eta(\mathbf{D},\cT_*) 
    \tbar v_*-v \tbar,
$$
where we have also used the equivalence~\eqref{eqn:equiv}.

Now let $v\in [u_-, u_+]\cap V_D(\cT),$ a short argument from the proof of Corollary 4.4 in~\cite{Cascon.J;Kreuzer.C;Nochetto.R;Siebert.K2008} 
gives
\begin{equation}
   \label{eqn:eta-reduct}
\eta^2(v,\cT_*) \leqs \eta^2(v,\cT \setminus \cM) + 2^{-(\ell/d)} \eta^2(v,\cM)
  = \eta^2(v,\cT) - \lambda \eta^2(v,\cM).
\end{equation}
Finally, the monotonicity properties
$\eta(\mathbf{D},\cT_*) \leqs \eta(\mathbf{D},\cT_0),$ combined with~\eqref{eqn:eta-reduct} yields the result.
\end{proof}

%%%%%%%%%%%%%%%%%%%%%%%%%%%%%%%%%%%%%%%%%%%%%%%%%%%%%%%%%%%%%%%%%%%%%%%%%%%%%%
\subsection{Quasi-Orthogonality for Nonlinear Problems}
  \label{sec:quasi-orthogonality}
%%%%%%%%%%%%%%%%%%%%%%%%%%%%%%%%%%%%%%%%%%%%%%%%%%%%%%%%%%%%%%%%%%%%%%%%%%%%%%

%It is somewhat surprising that the weak assumptions
%\eqref{eqn:pet_cont}--\eqref{eqn:pet_ah}
%used to derive abstract {\em a priori} error estimates
%for general Petrov-Galerkin approximations are nearly
%enough to establish the key quasi-orthogonality result
%that we need for proving convergence of AFEM
%in a general nonlinear setting.
%In addition to \eqref{eqn:pet_cont}--\eqref{eqn:pet_ah}
%we only need here the following additional assumption:
%The dimension of the subspace $X_h$ can be chosen sufficiently
%large so that the following inequality holds:
%\begin{equation}
%   \label{eqn:aubin-nitsche}
%\|u-u_h\|_{\calg{X}_{-}} \leqs C_{L} \sigma_h \tbar u-u_h \tbar,
%\end{equation}
%with the constant $\sigma_h$ as small as we desire.
%This is an abstraction of the idea of $L^2$-lifting:
%$$
%\|u-u_h\|_0 \leqs C_{L} h^s \|u-u_h\|_1,
%$$
%with $\calg{X}_{-} = L^2$ and $\sigma_h = h^s$,
%where $s$ depends on the regularity of the solution $u$,
%which is used in \cite{Mekchay.K;Nochetto.R2005} to establish a quasi-orthogonality
%result for a general linear problem.
%Our result below is completely abstract; it recovers the
%linear variation established in \cite{Mekchay.K;Nochetto.R2005} as a special
%case, and covers our more general nonlinear setting.

Following~\cite{HTZ08a}, we now
establish a quasi-orthogonality result that represents the last technical
result needed to generalize the convergence framework from
\cite{Cascon.J;Kreuzer.C;Nochetto.R;Siebert.K2008} 
to the nonlinear case.
\begin{lemma}[Quasi-orthogonality]
  \label{L:quasi-orthog}
	Let $u$ be the exact solution to equation \eqref{eqn:pbe_weak},  and $u_h$ be the solution to \eqref{eqn:disc} on a partition $\cT_{h}$ which satisfies the conditions in Assumptions~A1 and~A2. Assume that there exist a $\sigma_h>0$ with $\sigma_h\to 0$ as $h\to 0$ such that 
\begin{equation}
\label{eqn:aubin-nitsche}
	\|u-u_h\|_2 \leqs \sigma_h\|\nabla u-\nabla u_h\|_2,
\end{equation}
Then there exists a constant $C^*>0$, 
such that for sufficiently small $h,$ we have 
\begin{equation}
\tbar u-u_h \tbar^2 \leqs \Lambda_h \tbar u-u_H \tbar^2 
   - \tbar u_h - u_H \tbar^2,
\label{eqn:quasi-orthog}
\end{equation}
where $\Lambda_h = (1 - C^* \sigma_h K)^{-1}>0$ with $K = \sup_{\chi\in [u_-, u_+]}\|\kappa^2\cosh(\chi)\|_{\infty}.$
\end{lemma}

\begin{proof}
We compute the energy norm:
\begin{align*}
\tbar u-u_H \tbar^2
    &= a(u - u_H,u-u_H)
     = a(u-u_h + u_h - u_H,u-u_h + u_h - u_H)
\\
    &= \tbar u-u_h\tbar^2 + \tbar u_h-u_H \tbar^2
       + 2a(u-u_h,u_h-u_H).
\end{align*}
By the definition of $u$ and $u_h,$  we have
$$
a(u-u_h,v_h) + (b(u) - b(u_h), v_h ) = 0,
   ~~~\forall v_h \in V_0(\cT_h).
$$
In particular, this holds for $v_h=u_h-u_H.$
By this relation, we obtain
\begin{equation*}
\tbar u-u_H \tbar^2
    = \tbar u-u_h \tbar^2 + \tbar u_h-u_H\tbar^2
       + 2(\kappa^2 (\sinh(u)-\sinh(u_h)),u_h-u_H).
\end{equation*}
Therefore, by Young's inequality and 
the assumption \eqref{eqn:aubin-nitsche} we have
\begin{align*}
|2( b(u) - b(u_h), u_h-u_H)|
  &\leqs 2\sup_{\chi\in [u_-, u_+]}\|\kappa^2\cosh(\chi)\|_{\infty} \|u-u_h\|_2 \|u_h-u_H\|_2\\
  &\leqs \delta \|u-u_h\|_2^2
      + \frac{K^2}{\delta}\|u_h-u_H\|_2^2
\\
  &\leqs \delta C(\epsilon)\sigma_h^2 \tbar u-u_h \tbar^2
      + \frac{K^2}{\delta m} \tbar u_h-u_H \tbar^2,
\end{align*}
for $\delta > 0$ to be chosen later, and $m$ the coercivity constant.
For $\sigma_h$ sufficient small, we have
$$
(1 - \delta C(\epsilon) \sigma_h^2) \tbar u-u_h \tbar^2
  \leqs \tbar u-u_H \tbar^2
    - \left( 1 - \frac{K^2}{\delta m} \right) \tbar u_h - u_H \tbar^2.
$$
Define now the constant $C^* = C(\epsilon) / \sqrt{m}$ and take
$
\delta = K/(\sqrt{C(\epsilon)} \sqrt{m} \sigma_h).
$
We assume $\sigma_h$ is sufficiently small so that
$
\delta C(\epsilon) \sigma_h^2 = C^* \sigma_h K < 1.
$
This gives~\eqref{eqn:quasi-orthog}
%$$
%\tbar u-u_h\tbar^2 \leqs \Lambda_h \tbar u - u_H \tbar^2 
%   - \tbar u_h - u_H\tbar^2,
%$$
with $\Lambda_h = (1 - C^* \sigma_h K)^{-1}$.
\end{proof}
We note \eqref{eqn:aubin-nitsche} can be established by ``Nitsche trick'' under regularity assumptions (cf. \cite{Mekchay.K2005}).
%%%%%%%%%%%%%%%%%%%%%%%%%%%%%%%%%%%%%%%%%%%%%%%%%%%%%%%%%%%%%%%%%%%%%%%%%%%%%%
\subsection{The Main Convergence Result for AFEM}
%%%%%%%%%%%%%%%%%%%%%%%%%%%%%%%%%%%%%%%%%%%%%%%%%%%%%%%%%%%%%%%%%%%%%%%%%%%%%%

To establish this result, we will 
follow~\cite{HTZ08a}
and use a combination of the frameworks in \cite{Cascon.J;Kreuzer.C;Nochetto.R;Siebert.K2008,Mekchay.K;Nochetto.R2005}
rather than from \cite{CHX06b}.
This is because these frameworks are the first to handle dependence
of the oscillation on the discrete solution itself.
The quasi-orthogonality result is explicit in \cite{Mekchay.K;Nochetto.R2005}, but somewhat
hidden in \cite{CHX06b}.
The framework in \cite{Cascon.J;Kreuzer.C;Nochetto.R;Siebert.K2008} uses only orthogonality rather than
quasi-orthogonality, but has a number of improvements over \cite{Mekchay.K;Nochetto.R2005}
and \cite{CHX06b} in several respects, including a one-pass algorithm
using only the residual indicator.

The previous sections focused on establishing some supporting results
involving a nesting of three spaces $X_H \subset X_h \subset X$, where
these were abstract spaces in some cases, or specific finite element
subspaces of $H^1$.
In what follows, 
we now consider the asymptotic sequence of finite element spaces 
produced by the AFEM algorithm, and will use the results of the previous
sections with the subscript $h$ in $u_h$ and other quantities replaced by 
an integer $k$ representing the current subspace generated at step $k$
of AFEM.
To simplify the presentation further, we also denote
$$
e_k = \tbar u-u_k \tbar,
\hspace{0.5cm}
E_k = \tbar u_k - u_{k+1} \tbar,
$$
$$
\eta_k = \eta(u_k,\cT_k),
\hspace{0.5cm}
\eta_k(\cM_k) = \eta(u_k,\cM_k),
\hspace{0.5cm}
\eta_0(\mathbf{D}) = \eta_0(\mathbf{D},\cT_0),
$$
where $\mathbf{D}$ represents the set of problem coefficients and nonlinearity.
We also denote $V_k:=V_D(\cT_k)$ for simplicity. 
The supporting results we need have been established 
in~\S\ref{sec:indicator-reduction} and~\S\ref{sec:quasi-orthogonality}.
\begin{theorem}[Contraction]
   \label{T:convergence}
Let $\{\cT_k, V_k, u_k \}_{k\geqs 0}$ be the
sequence of finite element meshes, spaces, and solutions,
respectively, produced by AFEM($\theta$,$\ell$) with
marking parameter $\theta \in (0,1]$ and bisection level $\ell \geqs 1.$
Let $\cT_{k}$ satisfy the conditions in Assumptions~A1 and~A2, and $h_0$ be sufficiently fine so that Lemma~\ref{L:quasi-orthog}
holds for $\{\cT_k, V_k, u_k\}_{k\geqs 0}.$
Then, there exist constants $\gamma > 0$ and $\alpha \in (0,1)$,
depending only on $\theta$, $\ell$, and the shape-regularity of
the initial triangulation $\cT_0$, such that
$$
\tbar u - u_{k+1} \tbar^2 + \gamma \eta_{k+1}^2
   \leqs
\alpha^2 \left( \tbar u - u_{k} \tbar^2 + \gamma \eta_{k}^2 \right).
$$
\end{theorem}
\begin{proof}
We combine the frameworks in \cite{Cascon.J;Kreuzer.C;Nochetto.R;Siebert.K2008,Mekchay.K;Nochetto.R2005}
using the quasi-orthogonality result in Lemma~\ref{L:quasi-orthog}
rather than the approach in \cite{CHX06b} for nonlinearities.
Our notation follows closely \cite{Cascon.J;Kreuzer.C;Nochetto.R;Siebert.K2008}.
The proof requires the following tools:
\begin{enumerateX}
\item D\"orfler marking property given in equation~\eqref{eqn:dorfler-property}.
\item The global upper-bound in Lemma~\ref{L:estimator-bounds}.
\item Estimator reduction Lemma~\ref{L:estimator-reduction}.
\item Quasi-orthogonality Lemma~\ref{L:quasi-orthog}.
\end{enumerateX}
In addition, some results above used indicator monotonicity
properties~\eqref{eqn:eta-monotonicity}--\eqref{eqn:osc-monotonicity}
and monotonicity of data $\eta_k(\bD).$
Starting with the quasi-orthogonality result in Lemma~\ref{L:quasi-orthog}
we have
$$
e_{k+1}^2 \leqs \Lambda_{k+1} e_k^2 - E_k^2,
$$
which gives
$$
e_{k+1}^2 + \gamma \eta_{k+1}^2
\leqs \Lambda_{k+1} e_k^2 - E_k^2 + \gamma \eta_{k+1}^2.
$$
Employing now Lemma~\ref{L:estimator-reduction} for some $\delta>0$
to be specified later we have
$$
e_{k+1}^2 + \gamma \eta_{k+1}^2
\leqs \Lambda_{k+1} e_k^2 - E_k^2
  + (1+\delta) \gamma [\eta_k^2 - \lambda \eta_k^2(\cM_k)]
  + (1+\delta^{-1}) \gamma \Lambda_1 \eta_0^2(\bD) E_k^2,
$$
where $\lambda\in (0,1)$ as defined in Lemma~\ref{L:estimator-reduction}.
Take now $\delta > 0$ sufficiently small so that we can ensure
$\gamma < 1$ by setting:
\begin{equation}
  \label{eqn:gamma}
0 < \gamma = \gamma(\delta) = \frac{\delta}{(1+\delta)\Lambda_1 \eta_0^2(\bD)}
       < 1.
\end{equation}
Using \eqref{eqn:gamma} in the last term leads to
$$
e_{k+1}^2 + \gamma \eta_{k+1}^2
\leqs \Lambda_{k+1} e_k^2
  + (1+\delta) \gamma \eta_k^2
  - (1+\delta) \lambda \gamma \eta_k^2(\cM_k).
$$
We now use the marking strategy in equation~\eqref{eqn:dorfler-property} 
to give
\begin{equation}
   \label{eqn:eta}
e_{k+1}^2 + \gamma \eta_{k+1}^2
\leqs \Lambda_{k+1} e_k^2
  + (1+\delta) \gamma \eta_k^2
  - (1+\delta) \lambda \gamma \theta^2 \eta_k^2.
\end{equation}
To allow for simultaneous reduction of the error and indicator,
we follow \cite{Cascon.J;Kreuzer.C;Nochetto.R;Siebert.K2008} and split the last term into
two parts using an arbitrary $\beta \in (0,1)$:
$$
e_{k+1}^2 + \gamma \eta_{k+1}^2
\leqs \Lambda_{k+1} e_k^2 + (1+\delta) \gamma \eta_k^2
  - \beta (1+\delta) \lambda \gamma \theta^2 \eta_k^2
  - (1-\beta) (1+\delta) \lambda \gamma \theta^2 \eta_k^2.
$$
We now the first and third terms using the upper bound from
Lemma~\ref{eqn:upper} and the expression for $\gamma$
in~\eqref{eqn:gamma}, and combine the second and fourth terms as well, giving:
$$
e_{k+1}^2 + \gamma \eta_{k+1}^2
\leqs
\left( \Lambda_{k+1} - \frac{\beta \delta \lambda \theta^2}
                        {C_1 \Lambda_1 \eta_0^2(\bD)} \right)
   e_k^2
+ (1+\delta)(1-(1-\beta)\lambda \theta^2)
   \eta_k^2.
$$
This can be written in the form
$$
e_{k+1}^2 + \gamma \eta_{k+1}^2
\leqs
\alpha_1^2(\delta,\beta) e_k^2
+ \gamma \alpha_2^2(\delta,\beta) \eta_k^2,
$$
where
$$
\alpha_1^2(\delta,\beta) =
\Lambda_h - \delta \left[\frac{\beta \lambda \theta^2}
                        {C_1 \Lambda_1 \eta_0^2(\bD)} \right],
\hspace*{0.5cm}
\alpha_2^2(\delta,\beta) =
(1+\delta)(1-(1-\beta)\lambda \theta^2).
$$
By Lemma~\ref{L:quasi-orthog}, we have
$\Lambda_{k+1} = (1 - C^* \sigma_{k+1} K)^{-1}$ with $\sigma_{k+1}:=\sigma_{h_{k+1}},$
so that
$$
\alpha_1^2(\delta,\beta) = \frac{1}{1 - C^* \sigma_{k+1} K}
- \delta \left[\frac{\beta \lambda \theta^2}{C_1 \Lambda_1 \eta_0^2(\bD)}\right].
$$
By the assumptions in Lemma~\ref{L:quasi-orthog},
we can take the initial mesh so that $\sigma_{k+1} \geqs 0$ is as small as
we desire, or that $\Lambda_{k+1}$ is as close to one as we desire.
Therefore, we can simultaneously pick $\sigma_{k+1} > 0$ and $\delta > 0$
sufficiently small so that $\alpha_1^2 < 1$.
Either this choice of $\delta > 0$ ensures $\alpha_2^2 < 1$ as well,
or we further reduce $\delta$ so that:
$$
\alpha^2 = \max\{\alpha_1^2,\alpha_2^2\} < 1.
$$
This completes the proof.
\end{proof}

%% file: meshgen.tex
\section{Feature-Preserving Mesh Generation for Biomolecules} \label{sec:mesh}

Mesh generation from a molecule is one of the important components
in finite element modeling of a biomolecular system. There are two
primary ways of constructing molecular surfaces: one is based on
the `hard sphere' model~\cite{Rich77} and the other is based on
the level set of a `soft' Gaussian function~\cite{GrPi95}. In
the first model, a molecule is treated as a collection of `hard'
spheres with different radii, from which three types of surfaces
can be extracted: {\it van der Waals surface}, {\it solvent
accessible surface}, and {\it solvent excluded surface}
\cite{Conn83,GrPi95,LeRi71,Rich77}. The molecular surface can be
represented analytically by a list of seamless spherical patches
\cite{Conn83,ToAb96} and triangular meshes can be generated using
such tools as {\it MSMS}~\cite{SOS96}. In contrast, the `soft'
model treats each atom as a Gaussian-like smoothly decaying scalar
function in $\mathbb{R}^3$~\cite{Blin82,DuOl93,GrPi95}. The molecular
surfaces are then approximated by appropriate level sets (or
iso-surfaces) of the total of the Gaussian functions~\cite{Blin82,DuOl93}.
Because of its generality, robustness, and capability of producing
smooth surfaces, we will utilize the `soft' model (or level set
method) in our molecular mesh generation.

We now briefly outline the algorithms of constructing
triangular and tetrahedral meshes from a molecule that is given by
a list of centers and radii for atoms (e.g., PQR files
\cite{DNMB04} or PDB files with radii defined by
users \cite{BWFG}). 
More details can be found in our earlier work \cite{YuHCM08}. 
Figure~\ref{MeshgenPipeline} shows the pipeline of our mesh
generation toolchain. Note that our tool can also
take as input an arbitrary 3D scalar volume or a triangulated
surface mesh that has very low quality.

\begin{figure}[!ht]
\begin{center}
   \includegraphics[width=0.75\textwidth]{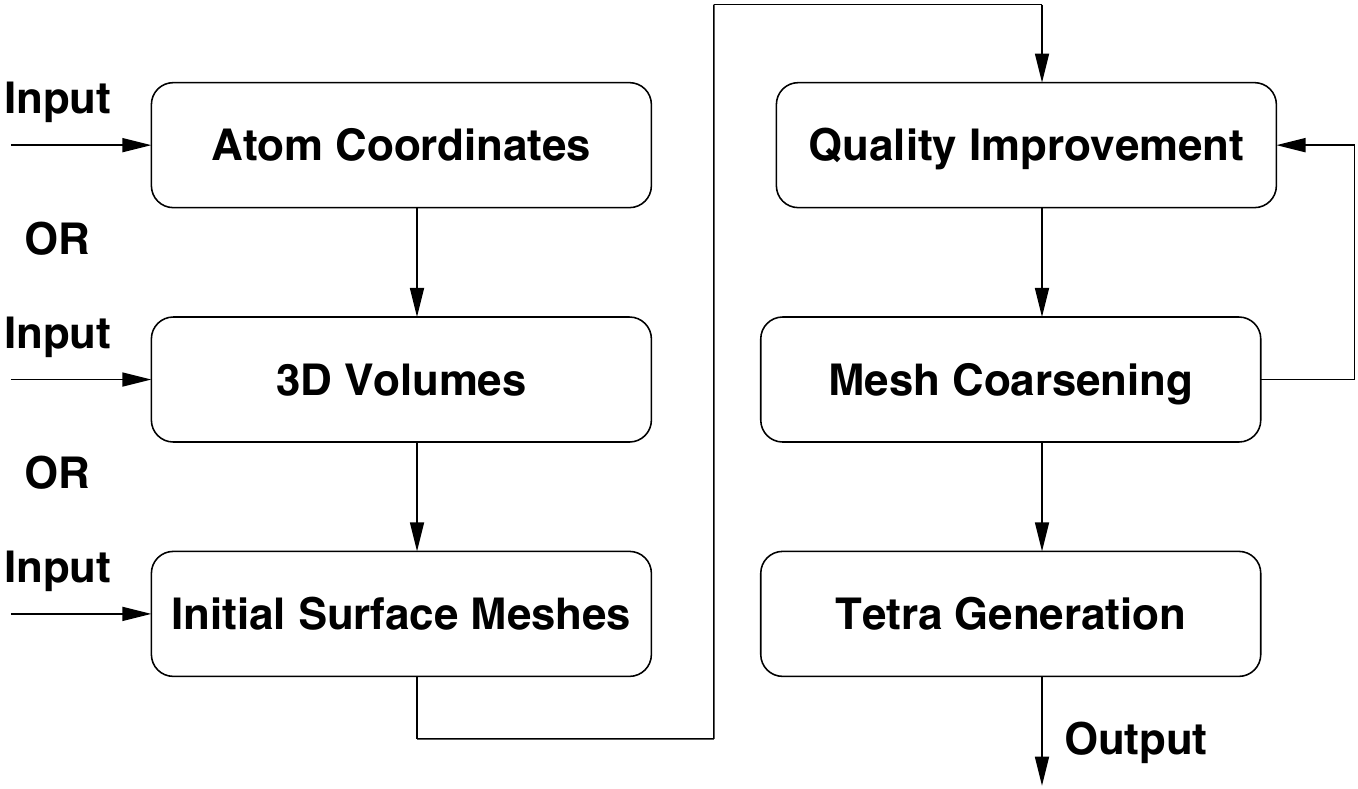}
\end{center}
\caption{Illustration of our mesh generation toolchain. The inputs
can be a list of atoms (with centers and radii), a 3D scalar
volume, or a user-defined surface mesh. The latter two can be
thought of as subroutines of the first one.}
\label{MeshgenPipeline}
\end{figure}

%%%%%%%%%%%%%%%%%%%%%%%%%%%%%%%%%%%%%%%
\subsection{Molecular Surface Generation}\label{subsec:generation}
In our mesh generation toolchain, a molecular surface mesh is
defined by a level set of the Gaussian kernel function computed
from a list of atoms (represented by centers ${\bf c}_i$ and radii
$r_i$) in a molecule as follows \cite{Blin82,GrPi95,ZXB06}:

\begin{equation}
F(\bf x) = \sum_{i=1}^{N}e^{B_i(\frac{\|{\bf x}-{\bf
c}_i\|^2}{r_i^2}-1)}, \label{eq:gaussian-blur1}
\end{equation}
where the negative parameter $B_i$ is called the {\it blobbyness}
that controls the spread of characteristic function of each atom.
The blobbyness is treated in our work as a constant parameter
(denoted by $B_0$) for all atoms. Our experiments on a number of
molecules show that the blobbyness at $-0.5$ produces a good
approximation for molecular simulations.

Given the volumetric function $F(\bf x)$, the surface (triangular)
mesh is constructed using the marching cube method \cite{LoCl87}.
Figure~\ref{MeshgenExample}(A) shows an example of the isosurface
extracted using this method. From this example, we can see that:
(a) the isosurfacing technique can extract very smooth surfaces,
but (b) many triangles are extremely ``sharp'', which can cause
poor approximation quality in finite element analysis. In
addition, meshes generated by isosurfacing techniques are often
too dense. Therefore, improving mesh quality yet keeping the
number of mesh elements small are two important 
issues that we will address in this section.

\begin{figure}[!ht]
\begin{center}
   \includegraphics[width=0.98\textwidth]{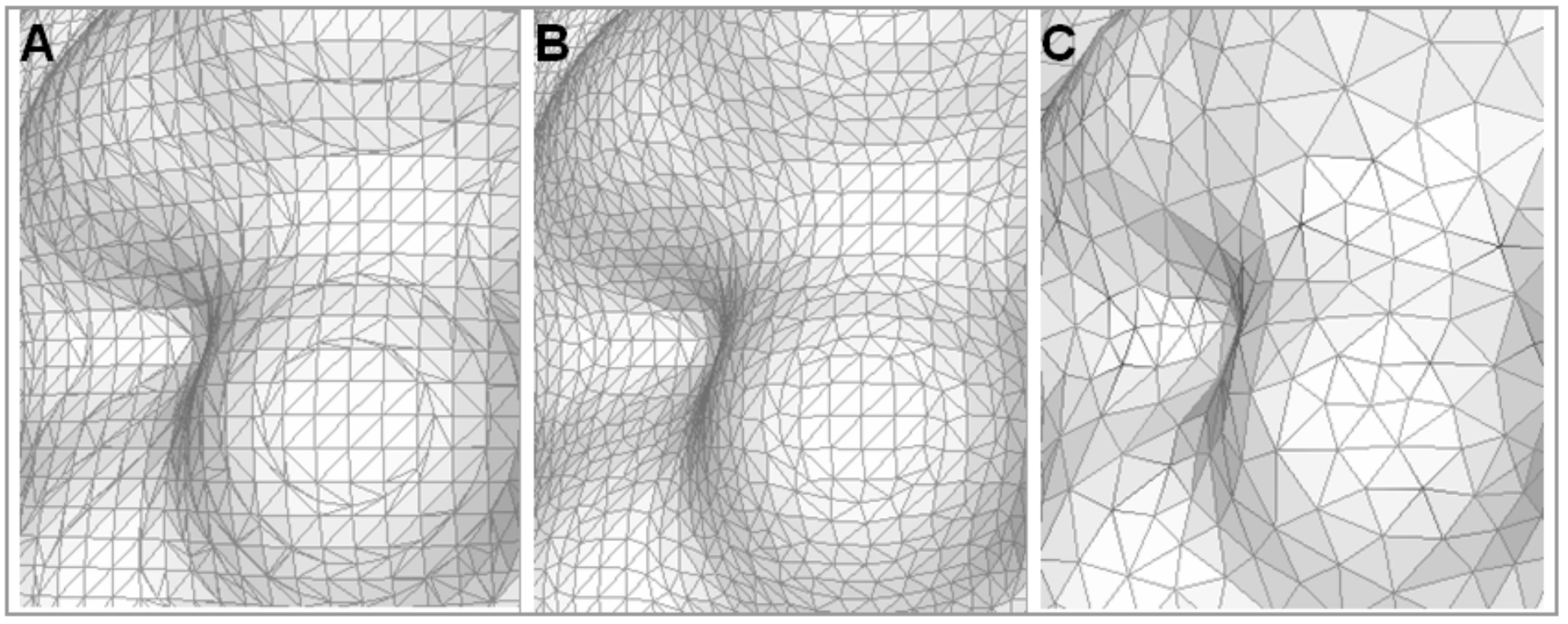}
\end{center}
\caption{Illustration of the surface generation and
post-processing. (A) A 3D volume is first generated using the
Gaussian blurring approach (equation \eqref{eq:gaussian-blur1})
from the molecule (PDB: 1CID). Shown here is part of the surface
triangulation by the marching cube method. (B) The surface mesh
after two iterations of mesh quality improvements. (C) After
coarsening, the mesh size becomes about seven times smaller than
the original one. The mesh is also smoothed by the normal-based
technique.} \label{MeshgenExample}
\end{figure}

%%%%%%%%%%%%%%%%%%%%%%%%%%%%%%%%%%%%%%%
\subsection{Surface Mesh Improvement and Decimation}\label{subsec:processing}
Surface mesh post-processing includes quality improvement and mesh
coarsening (decimation). The mesh quality can be improved by a combination 
of three major techniques: inserting or deleting vertices, swapping
edges or faces, and moving the vertices without changing the mesh
topology \cite{Freitag.L;Ollivier-Gooch.C1997}. The last one is the main strategy we use
to improve the mesh quality in our toolchain. For a surface mesh,
however, moving the vertices may change the shape of the surface.
Therefore, when we move the vertices, important features (e.g.,
sharp boundaries, concavities, holes, etc.) on the original
surface should be preserved as much as possible. To characterize
the important features on the surface mesh, we compute so-called
{\em local structure tensor} \cite{FeLi03,Weic98,Yu.Z;Bajaj.C2008} as
follows:
\begin{equation}
T({\bf v}) = \sum_{i=1}^{M'}\left(
\begin{array}{ccc}
    {n}_x^{(i)}{n}_x^{(i)} & {n}_x^{(i)}{n}_y^{(i)} & {n}_x^{(i)}{n}_z^{(i)} \\ \\
    {n}_y^{(i)}{n}_x^{(i)} & {n}_y^{(i)}{n}_y^{(i)} & {n}_y^{(i)}{n}_z^{(i)}  \\ \\
    {n}_z^{(i)}{n}_x^{(i)} & {n}_z^{(i)}{n}_y^{(i)} & {n}_z^{(i)}{n}_z^{(i)}
\end{array} \right),
\label{eq:struc_tensor}
\end{equation}
where $({n}_x^{(i)},{n}_y^{(i)},{n}_z^{(i)})$ is the \emph{normal vector}
of the $i^{th}$ neighbor of a vertex {\bf v} and $M'$ is the total
number of neighbors. The normal vector of a vertex is defined by
the weighted average of the normals of all its incident triangles.
The local structure tensor basically captures the principal axes
of a set of vectors in space. Let the eigenvalues of $T({\bf v})$
be $\lambda_1, \lambda_2, \lambda_3$ and $\lambda_1 \geq \lambda_2
\geq \lambda_3$. Then the local structure tensor can capture the
following features: (a) Spheres and saddles: $\lambda_1 \approx
\lambda_2 \approx \lambda_3
> 0$; (b) Ridges and valleys: $\lambda_1 \approx \lambda_2 \gg \lambda_3
\approx 0$; (c) Planes: $\lambda_1 \gg \lambda_2 \approx \lambda_3
\approx 0$.

The quality of a mesh can be improved by maximizing the minimal
angles. The angle-based method developed by Zhou and coauthors
\cite{ZhSh00} utilizes this idea by moving a vertex (denoted by
{\bf{x}}) towards the bisectors of the angles formed by adjacent
vertices on the surrounding polygon. This method works quite well
for 2D planar meshes and has been extended in \cite{XuNe2005} for
improving quadrilateral mesh quality as well. However, vertices on
a surface mesh can move with three degrees of freedom. If only the
angle criterion is considered, the surface mesh may become bumpy
and some molecular features may disappear. In other words, while the mesh quality is being improved, the
geometric features on a surface mesh should be preserved as much
as possible. To this end,
we take advantage of the local structure tensor by mapping the new
position $\bar{\bf{x}}$ generated by the angle-based method to
each of the eigenvectors of the tensor calculated at the original
position $\bf{x}$ and scaling the mapped vectors with the
corresponding eigenvalues. Let $\bf{e}_1, \bf{e}_2, \bf{e}_3$
denote the eigenvectors and $\lambda_1, \lambda_2, \lambda_3$ be
the corresponding eigenvalues of the local structure tensor valued
at $\bf{x}$. The modified vertex $\hat{\bf{x}}$ is calculated as
follows:

\begin{equation}
\hat{\bf{x}} = {\bf{x}} +
\sum_{k=1}^3\frac{1}{1+\lambda_k}((\bar{\bf{x}}-{\bf{x}})\cdot
{\bf{e}}_k) {\bf{e}}_k. \label{eq:angle-mapping}
\end{equation}
The use of eigenvalues as a weighted term in the above equation is
essential to preserve the features (with high curvatures) and to
keep the improved surface mesh as close as possible to the
original mesh by encouraging the vertices to move along the
eigen-direction with small eigenvalues (or in other words, with
low curvatures). Figure~\ref{MeshgenExample}(B) shows the surface
mesh after quality improvement, compared to the original mesh as
shown in Figure~\ref{MeshgenExample}(A). Before quality improvement,
the minimal and maximal angles are $0.02^{\circ}$ and
$179.10^{\circ}$ respectively. These angles become $14.11^{\circ}$
and $135.65^{\circ}$ after the improvement (two iterations).

The surface meshes extracted by isocontouring techniques (e.g.,
the marching cube method) often contain a large number of elements
and are nearly uniform everywhere. To reduce the computational
cost, adaptive meshes are usually preferred where fine meshes only
occur in regions of interest. The idea of mesh coarsening in our pipeline is
straightforward $-$ delete a node and its associated edges, and
then re-triangulate the surrounding polygon. The local structure
tensor is again used as a way to quantify the features. Let ${\bf
x}$ denote the node being considered for deletion and the
neighboring nodes be ${\bf v}_i, i = 1, \cdots, M$, where $M$ is
the total number of the neighbors. The maximal length of the
incident edges at ${\bf x}$ is denoted by $L({\bf x}) =
\max_{i=1}^M\{d({\bf x}, {\bf v}_i)\}$ where $d(\cdot, \cdot)$ is the
Euclidean distance. Apparently $L({\bf x})$ indicates the
sparseness of the mesh at ${\bf x}$. Let $\lambda_1({\bf x}),
\lambda_2({\bf x}), \lambda_3({\bf x})$ be the eigenvalues of the
local structure tensor calculated at ${\bf x}$, satisfying
$\lambda_1({\bf x}) \geq \lambda_2({\bf x}) \geq \lambda_3({\bf
x})$. Then the node ${\bf x}$ is deleted if and only if the
following condition holds:

\begin{equation}
L({\bf x})^\alpha\left(\frac{\lambda_2({\bf x})}{\lambda_1({\bf
x})}\right)^\beta < T_0, \label{eq:mesh-coarsening}
\end{equation}
where $\alpha$ and $\beta$ are chosen to balance between the
sparseness and the curvature of the mesh. In our experiments, they
both are set as $1.0$ by default. The threshold $T_0$ is
user-defined and also dependent on the values of $\alpha$ and
$\beta.$ When $\alpha$ and $\beta$ are fixed, larger $T_0$
will cause more nodes to be deleted. For the example in
Figure~\ref{MeshgenExample}(C), the coarsened mesh consists of
$8,846$ nodes and $17,688$ triangles, about seven times smaller
than the mesh as shown in Figure~\ref{MeshgenExample}(B).

Mesh coarsening can greatly reduce the mesh size to a
user-specified order. However, the nodes on the ``holes'' are
often not co-planar; hence the re-triangulation of the ``holes''
often results in a bumpy surface mesh. The bumpiness can be
reduced or removed by smoothing the surface meshes. We employ the
idea of anisotropic vector diffusion \cite{PeMa90,YuBa04} and
apply it to the normal vectors of the surface mesh being
considered. This normal-based approach turns out to preserve sharp
features and prevent volume shrinkages \cite{ChCh05} better than
the traditional vertex-based approach. Figure~\ref{MeshgenExample}(C)
shows the result after the mesh coarsening and normal-based mesh
smoothing.

%%%%%%%%%%%%%%%%%%%%%%%%%%%%%%%%%%%%%%%
\subsection{Tetrahedral Mesh Generation}\label{subsec:tetgen}
Once the surface triangulation is generated with good quality,
{\it Tetgen} \cite{Si04,SiGa05} can produce tetrahedral meshes
with user-controlled quality. Besides the triangulated surface,
our toolchain will have three other outputs for a given molecule:
the interior tetrahedral mesh, the exterior tetrahedral mesh, and
both meshes together. For the interior tetrahedral mesh, we force
all atoms to be on the mesh nodes. The exterior tetrahedral mesh
is generated between the surface mesh and a bounding sphere whose
radius is set as about $40$ times larger than the size of the
molecule being considered. Figure~\ref{MeshgenTest} demonstrates an
example of mesh generation on the mouse Acetylcholinesterase
(mAChE) monomer.

\begin{figure}[!ht]
\begin{center}
   \includegraphics[width=0.195\textwidth]{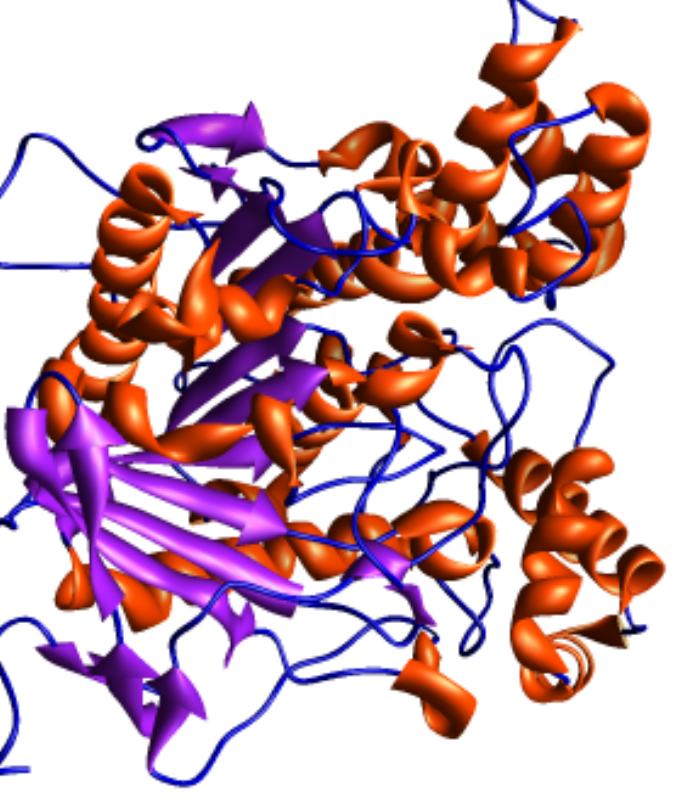}
   \includegraphics[width=0.260\textwidth]{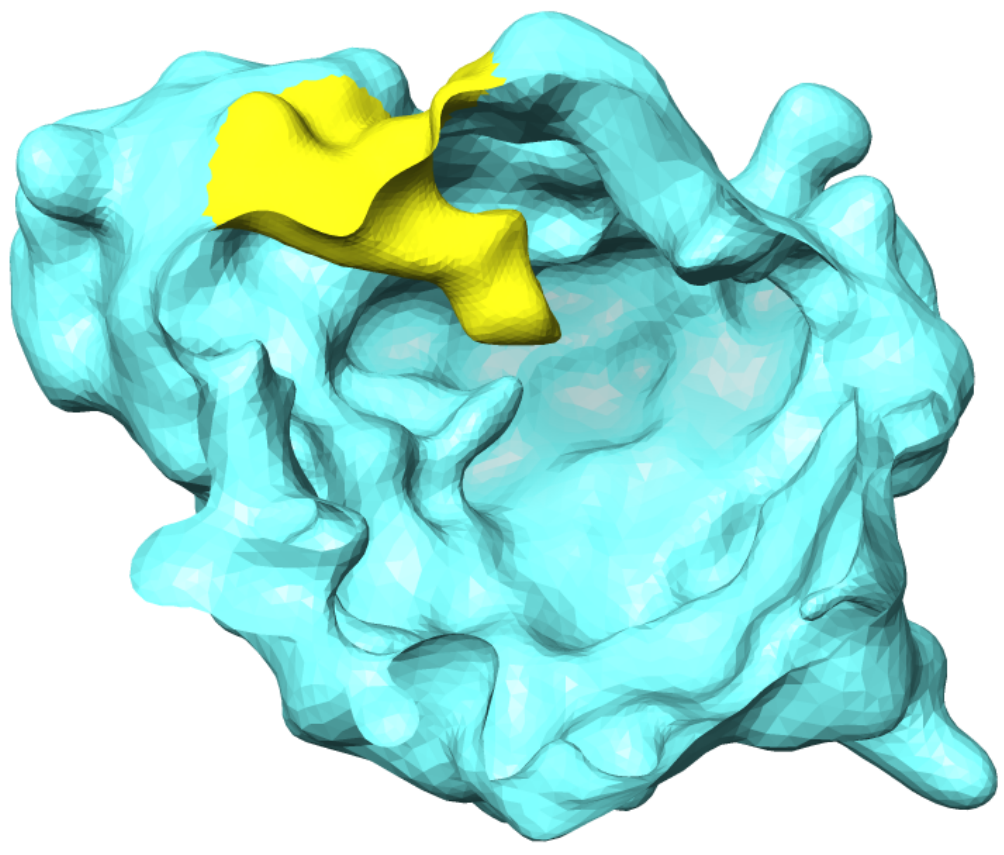}
   \includegraphics[width=0.260\textwidth]{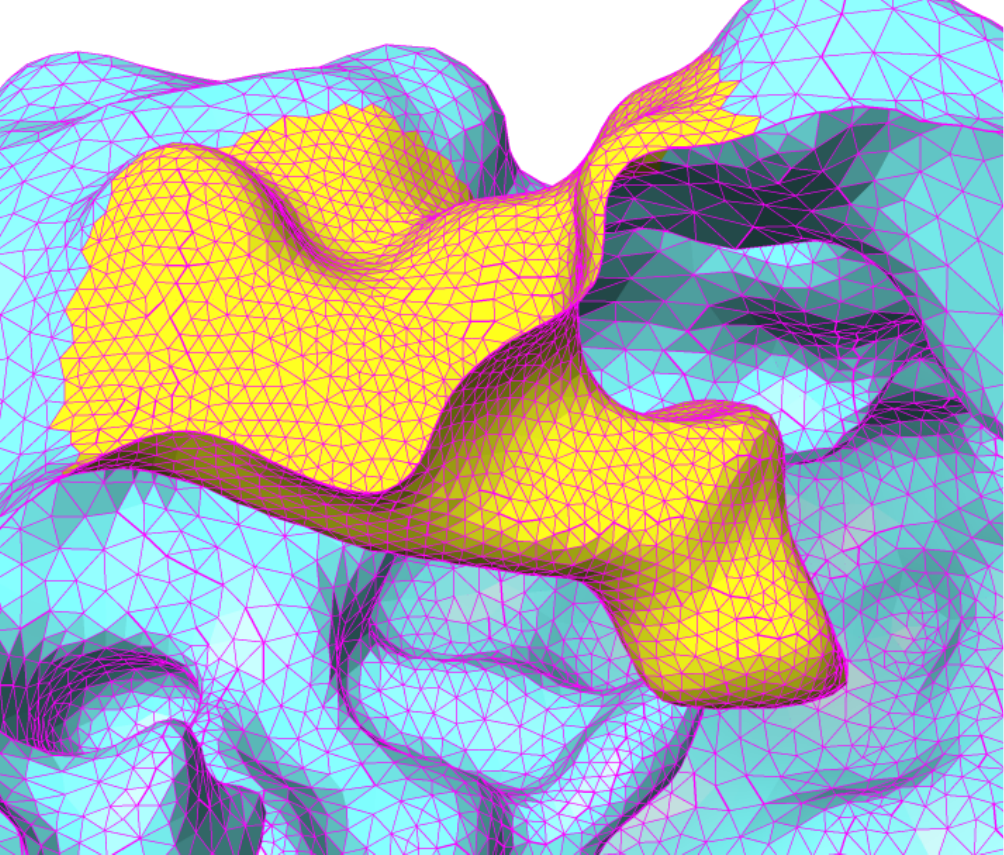}
   \includegraphics[width=0.260\textwidth]{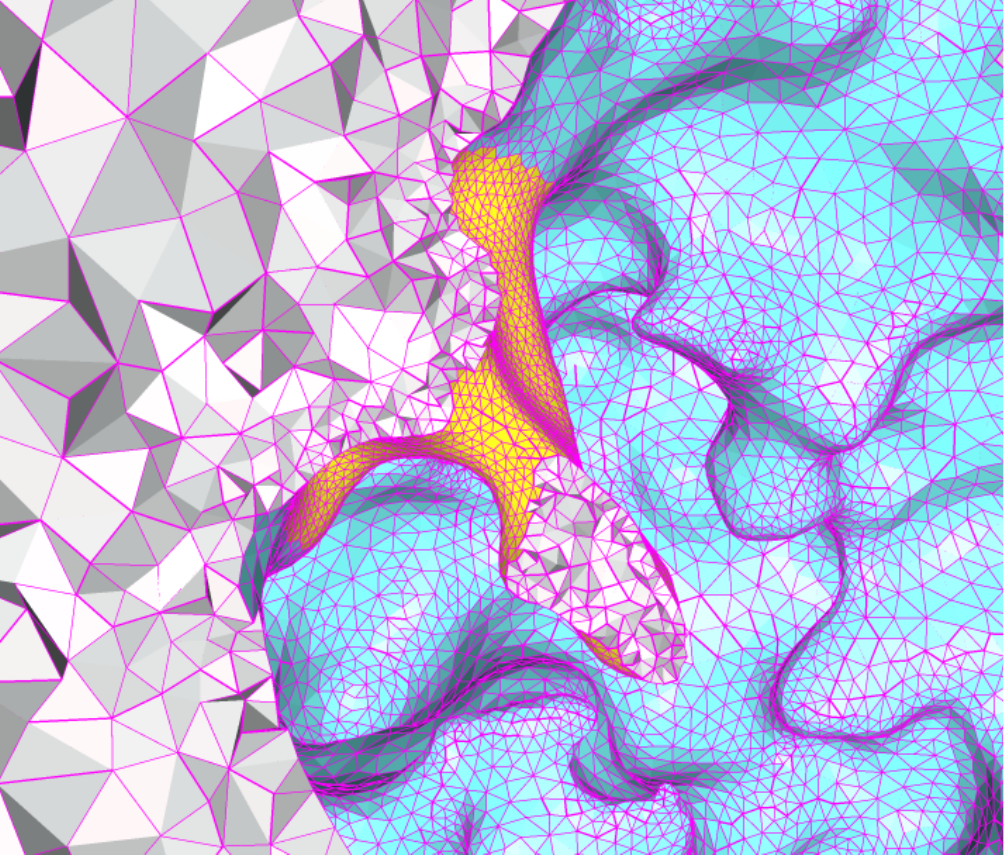} \\
\hspace*{-0.3cm} (A)
\hspace*{2.8cm} (B)
\hspace*{3.2cm} (C)
\hspace*{3.2cm} (D)
\end{center}
\caption{Illustration of biomolecular mesh generation. (A) The PDB
structure of the mouse Acetylcholinesterase (mAChE) monomer. (B)
The surface mesh generated by our approach. The active site is
highlighted in yellow. (C) A closer look at the mesh near the active
site. (D) The tetrahedral mesh between the molecular surface and
the bounding sphere (not shown).} \label{MeshgenTest}
\end{figure}

%% file: numeg.tex
\section{Numerical Examples} \label{sec:numerical}

Two numerical examples with increasingly complexity of molecular surface are presented to show the stability 
of the decomposition scheme and the convergence of the adaptive algorithm. In both examples,
the Laplace equation for harmonic component is solved with finite element method. The gradient of
the harmonic component is then computed and supplied for calculating the interface conditions of the regularized
Poisson-Boltzmann equation. It is also possible to directly compute the harmonic component and its gradient from 
the solution representation for the Poisson equation of the harmonic component via surface integrals on the
molecular surface.

{\bf Example 1.} The first numerical example is devoted to the comparison of two decomposition schemes discussed in the
subsection \ref{sec:pbe}. We use the model problem in Example \ref{eg:born_ion_1} because it admits an analytical
solution for comparison. The computational domain is chosen to be a sphere with radius $r=5$\AA. 
Figure \ref{fig:born_ion_1} plots the computed regular potential component and the
full potential from the first decomposition scheme as well as their relative errors with respect to the analytical
solutions, respectively. 
Chart B shows that the finite element solution of this regular component has
an relative error below 3\% over the entire domain. Because of the large magnitude of this regular potential, the
absolute error is considerably large, see Chart A and in particular Chart C, 
where the analytical singular component is added to get the full potential. The amplification of the 
relative error as analyzed in section \ref{sec:pbe} is seen from a comparison of 
Chart B and Chart D. This confirms that the first 
decomposition scheme is numerically unstable.
\begin{figure}[!ht]
\begin{center}
   \includegraphics[width=0.45\textwidth]{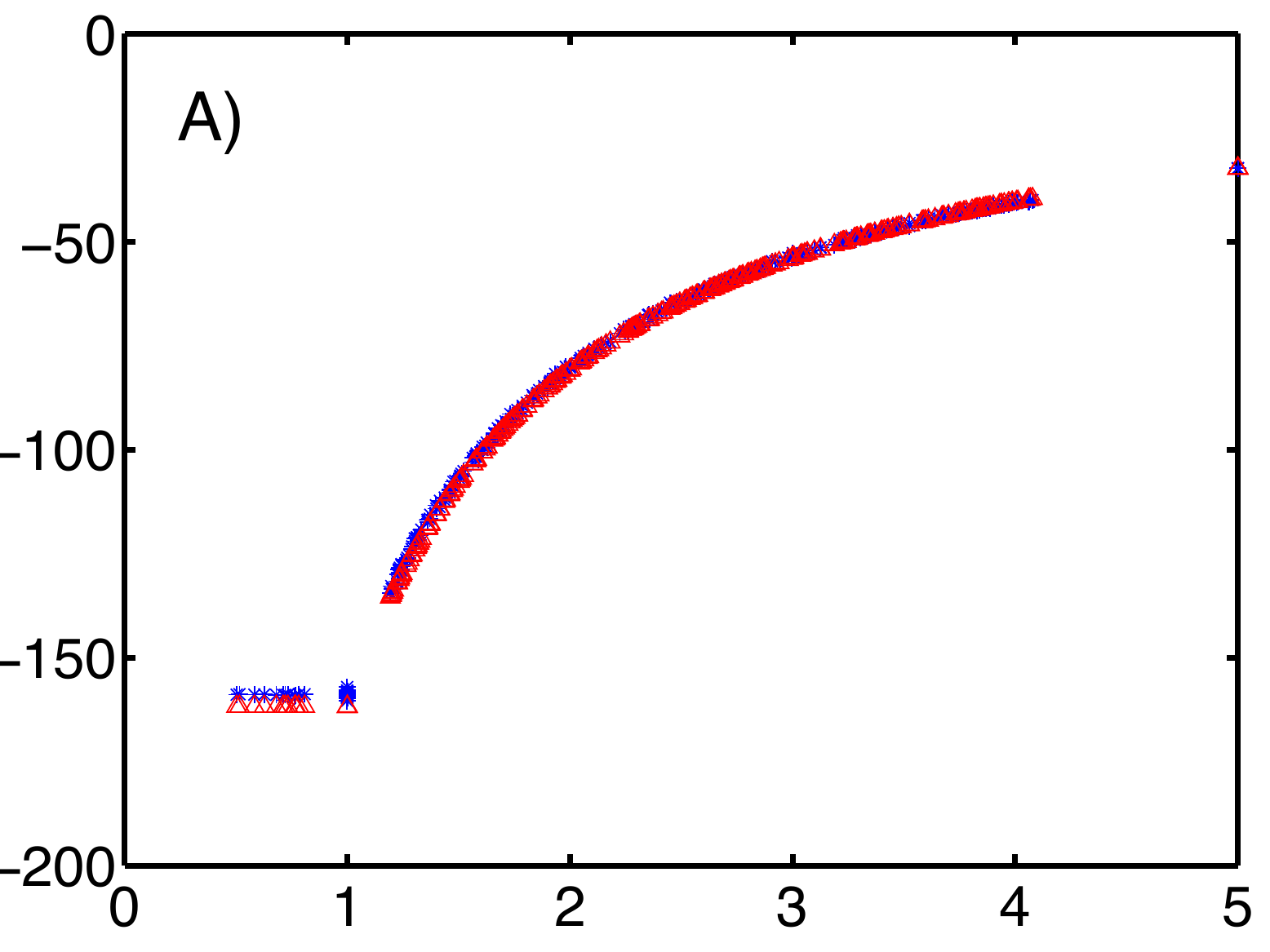}
   \includegraphics[width=0.45\textwidth]{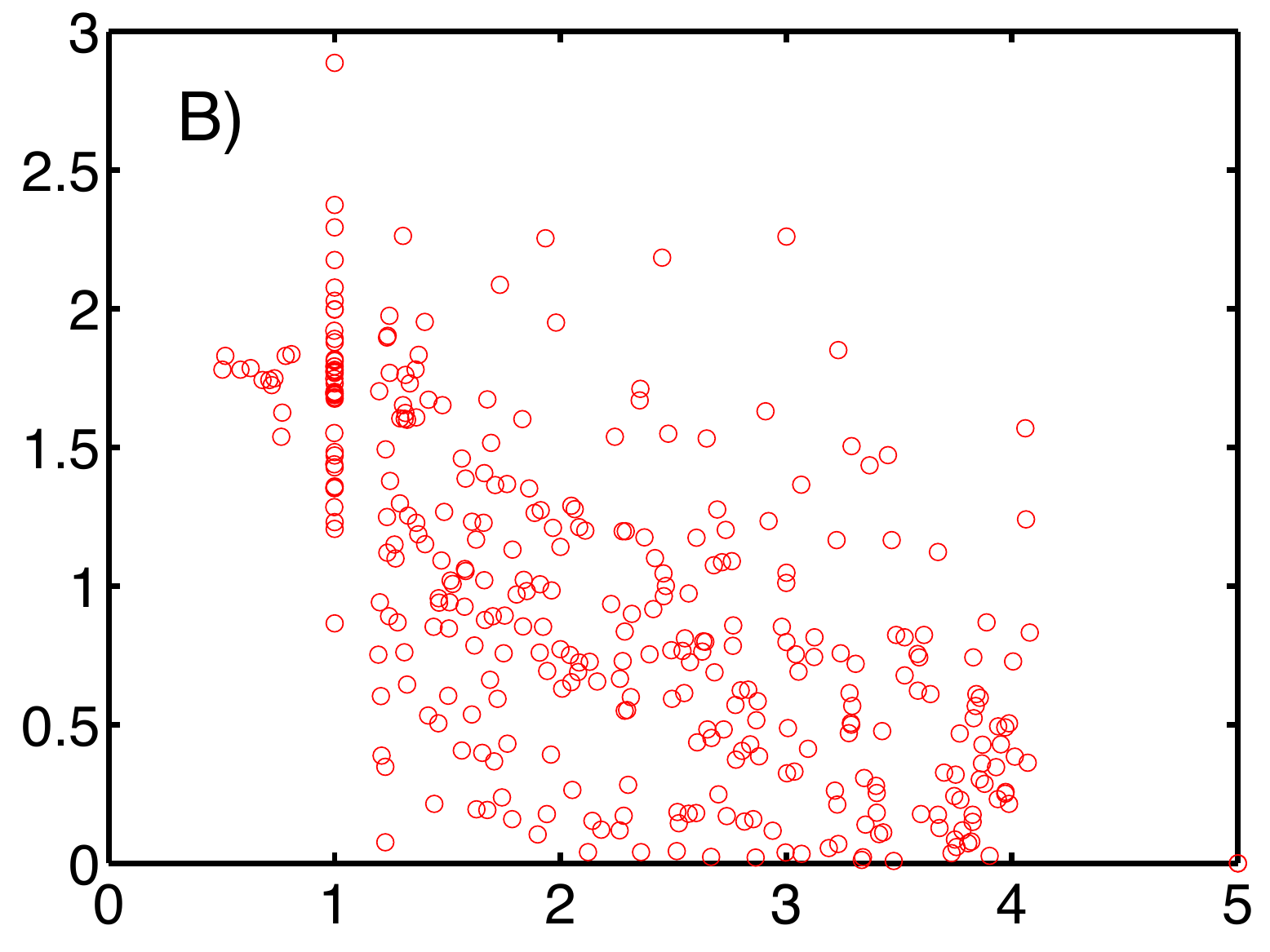} \\
   \includegraphics[width=0.45\textwidth]{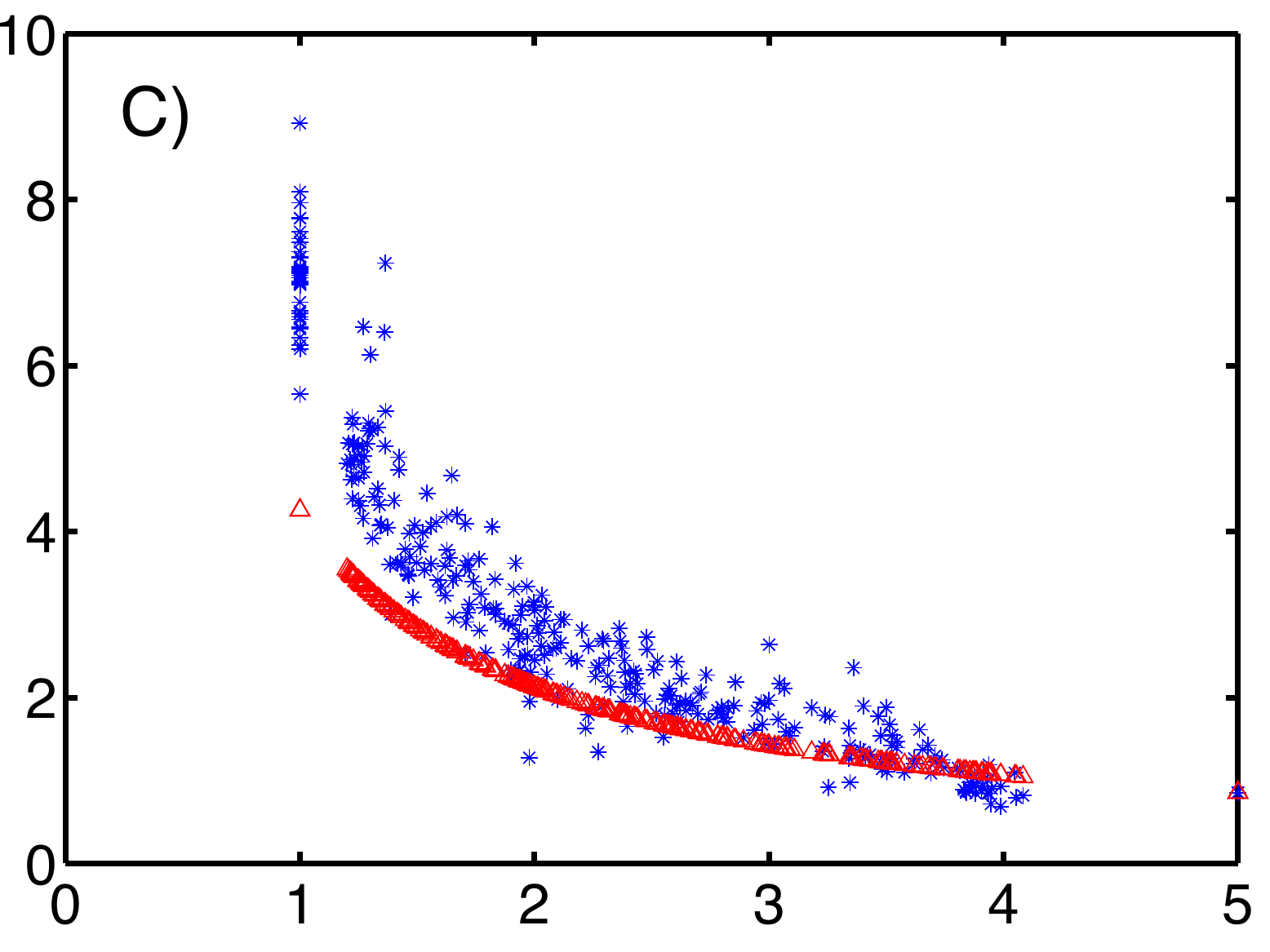}
   \includegraphics[width=0.45\textwidth]{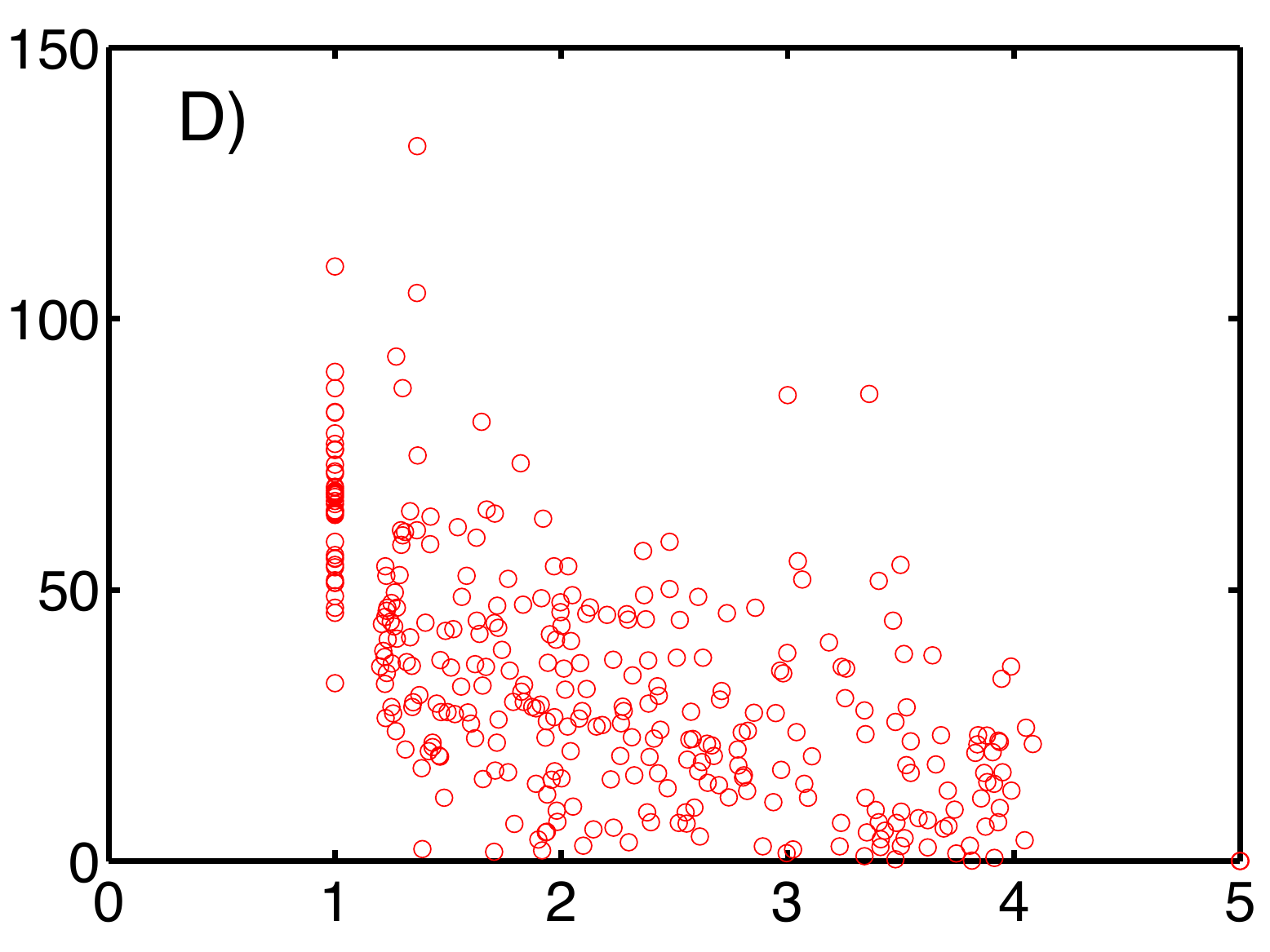} \\
\end{center}
\caption{
Solution of the Poisson-Boltzmann equation via the first decomposition scheme.
(A): Computed regular component $u^r$ of the electrostatic potential (blue) versus the analytical solution (red).
(B): Relative error in percentage of computed regular component $u^r$ of electrostatic potential.
(C): Computed full electrostatic electrostatic potential (blue) versus the analytical solution (red).
(D): Relative error in percentage of the computed full electrostatic potential.
}
\label{fig:born_ion_1}
\end{figure}

The numerical solutions via the second decomposition scheme demonstrate the desirable numerical
stability, as shown in Figure~\ref{fig:born_ion_2}. The regular potential $u^r$ in Chart A is solved
with the same mesh for Figure~\ref{fig:born_ion_1}, and shows a very
good agreement with the analytical solution. The relative error is well below 1.5\% over the entire
domain, and is well below 0.1\% in the interior of the molecule. Compared to Figure~\ref{fig:born_ion_1},
it is seen that the magnitude of the regular component of the stable decomposition is much 
smaller. Because the harmonic and regular components are both solved numerically, it is worthwhile to
examine the summation of these two numerical solutions and compare the total with the exact solution;
this is plotted in Chat B. The discontinuity indicates that the decomposition is only applied 
inside the biomolecule, and that the harmonic component is much larger than the regular 
component. This further suggests that the overall relative numerical error inside the biomolecule 
will be larger than that in the solvent region. Interesting enough, most intermolecular electrostatic
interactions are occurred through the solvent, and thus the stable decomposition can still provide
the electrostatic potential of high fidelity for describing these interactions.
We then refine this 
mesh globally by bisecting all the edges; the relative error is reduced 0.5\% in most of the domain except in the 
vicinity of the dielectric interface where the error does not show a noticeable decrease, see Chart D. 
This is because the middle point of an edge on the interface maybe not located on the interface and
therefore violates the assumption on the discretization of the molecular surface, thus 
the approximations to the interface and to the interface conditions are not improved
with this globally refinement. To satisfy this assumption we apply this global refinement first
and then move the middles points of all the interface edges back to the interface. This new 
refinement approach successfully scales down the numerical error near the interface, see Chart E.
\begin{figure}[!ht]
\begin{center}
   \includegraphics[width=0.45\textwidth]{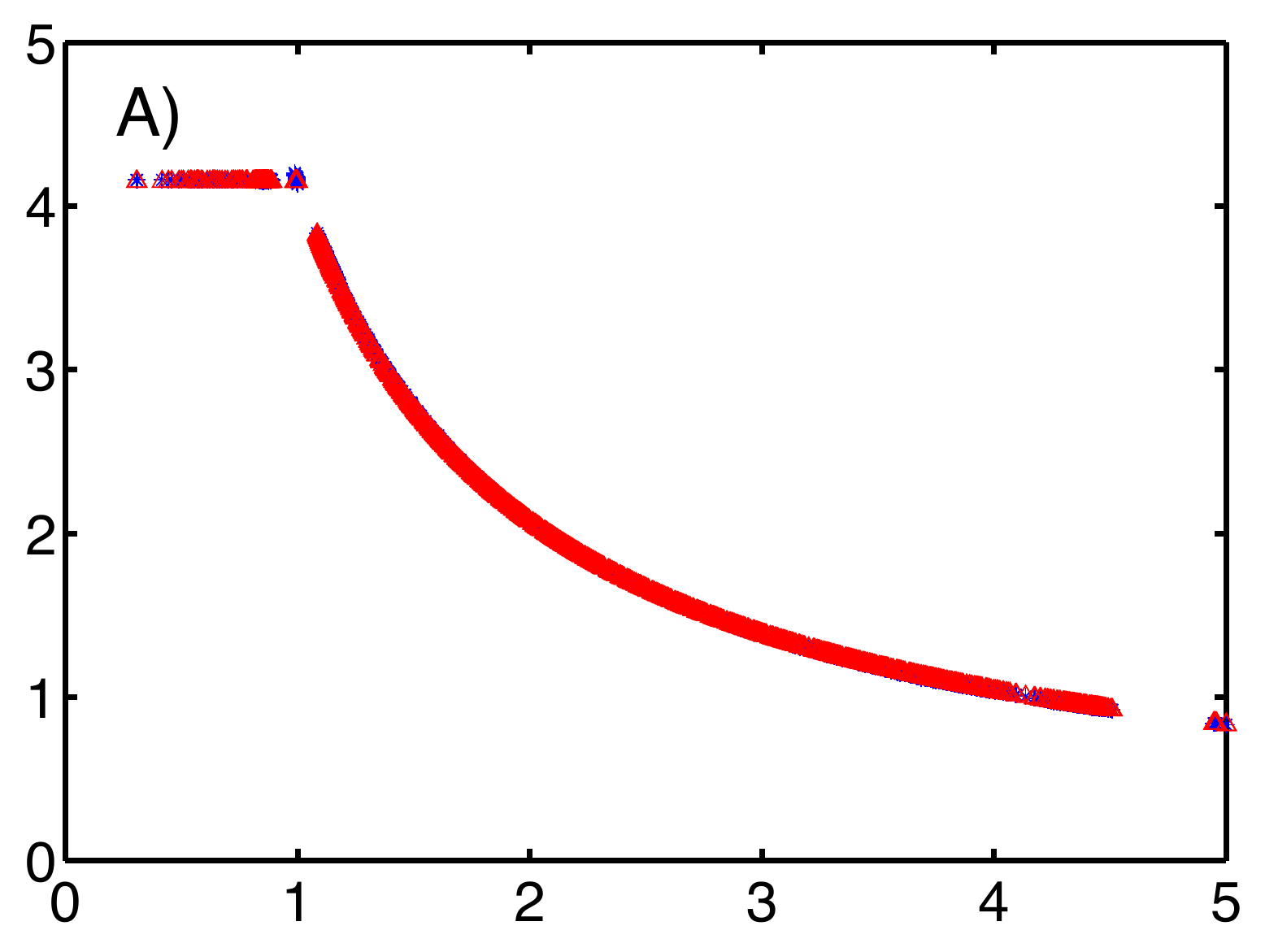}
   \includegraphics[width=0.45\textwidth]{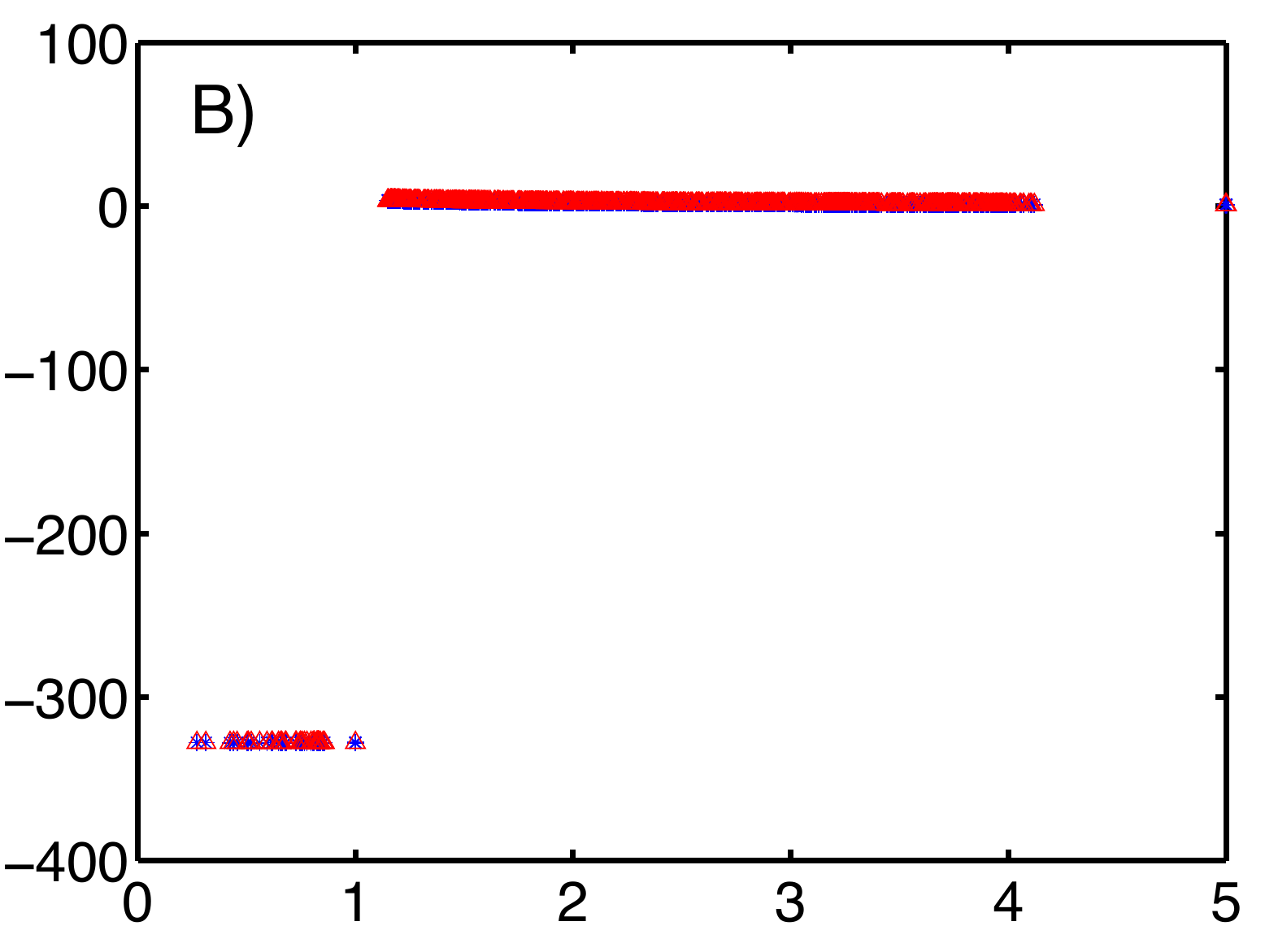} \\
   \includegraphics[width=0.45\textwidth]{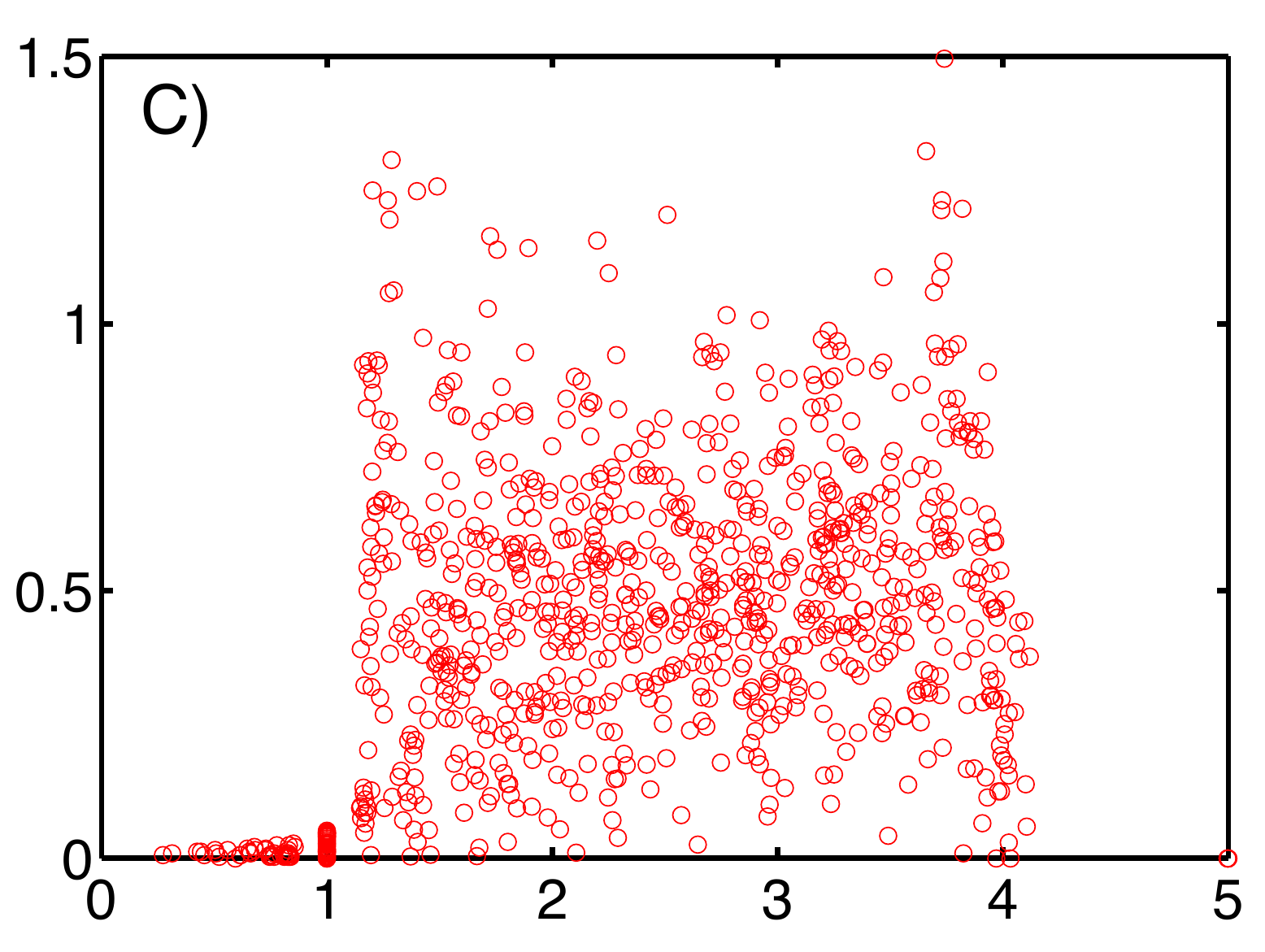}
   \includegraphics[width=0.45\textwidth]{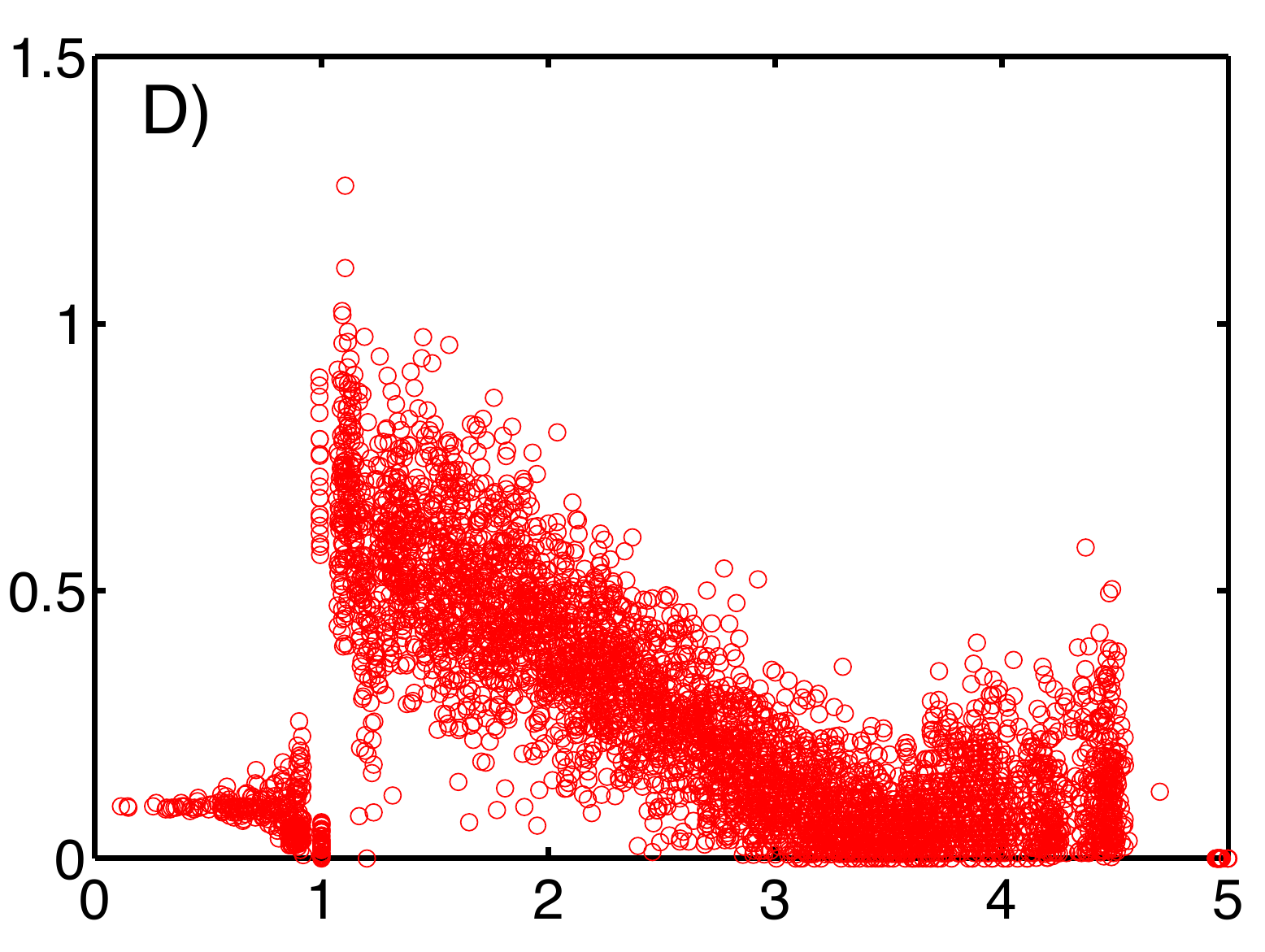} \\
   \includegraphics[width=0.45\textwidth]{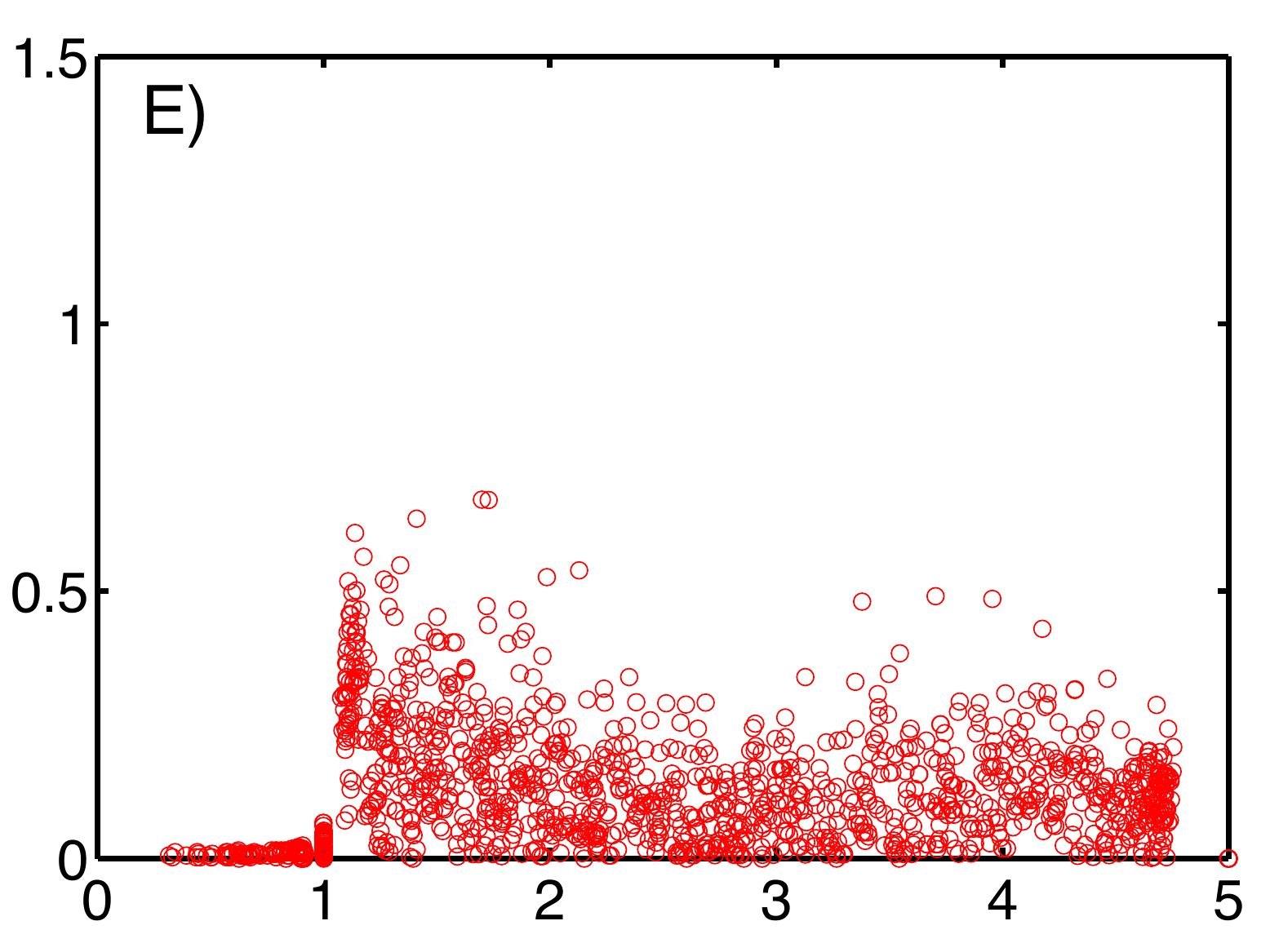} 
\end{center}
\caption{
Solution of the Poisson-Boltzmann equation via the second decomposition scheme. 
(A): Computed regular component $u^r$ of the electrostatic potential (blue) versus the analytical solution (red). 
(B): Computed regular component $u^r$ plus the harmonic component $u^h$ of the electrostatic potential (blue) versus the analytical solution (red).
(C): Relative error in percentage of the computed regular component of the electrostatic potential on an initial mesh.
(D): Relative error in percentage of the computed regular component of the electrostatic potential; globally refined mesh. 
(E): Relative error in percentage of the computed regular component of the electrostatic potential; mesh locally refined on molecular surface and the boundary.
}
\label{fig:born_ion_2}
\end{figure}

Both two decomposition schemes give rise to an elliptic interface problem, whose solution is of $C^0$ only and
can be appreciated from Chart A of Figures~\ref{fig:born_ion_1} and~\ref{fig:born_ion_2}. 

{\bf Example 2.} This second numerical experiment is conducted on an insulin 
protein \cite{WXH04_insulin} (PDB ID: 1RWE). 
This protein has two polypeptide chains, one has 21 amino acid residues and 
the other has 30 residues, and has 1578 atoms in total. 
Because there is no analytical solution available for accuracy assessment 
we solve the Poisson-Boltzmann equation on four progressively refined meshes 
and use the solution on the finest mesh as the reference to measure
the accuracy of other three solutions.
In Table~\ref{table:accuracy} we show the computed electrostatic solvation 
energy ($\Delta G_{ele}$) and the corresponding relative error in the 
solution ($e_{\Delta G_{ele}}$) for each solution.
This energy is defined as
\begin{eqnarray*}
\Delta G_{ele} = \frac{1}{2} \int_{\Omega} (\rho_{sol} - \rho_{vac}) \rho^f dx,
\end{eqnarray*}
where $\rho_{sol}$ is the electrostatic potential of the solvated molecule 
while $\rho_{vac}$ is the potential for the molecule in vacuum, where the 
dielectric constant is assumed to the same as the interior of the molecule.
It turns out that $\rho_{vac}$ is essentially the singular component inside 
molecule, and thus the solvation energy can be directly computed as
\begin{eqnarray*}
\Delta G_{ele} = \frac{1}{2} \int_{\Omega} (\rho^h + \rho^r) \rho^f dx.
\end{eqnarray*}
Assume that the solution on the finest mesh is convergent, we computed the 
relation error in the solution energy for three coarser meshes. 
The diminishing of this relative error confirms the convergence of the 
numerical method for computing electrostatics of realistic biomolecules.
\begin{figure}[!ht]
\begin{center}
   \includegraphics[width=0.45\textwidth]{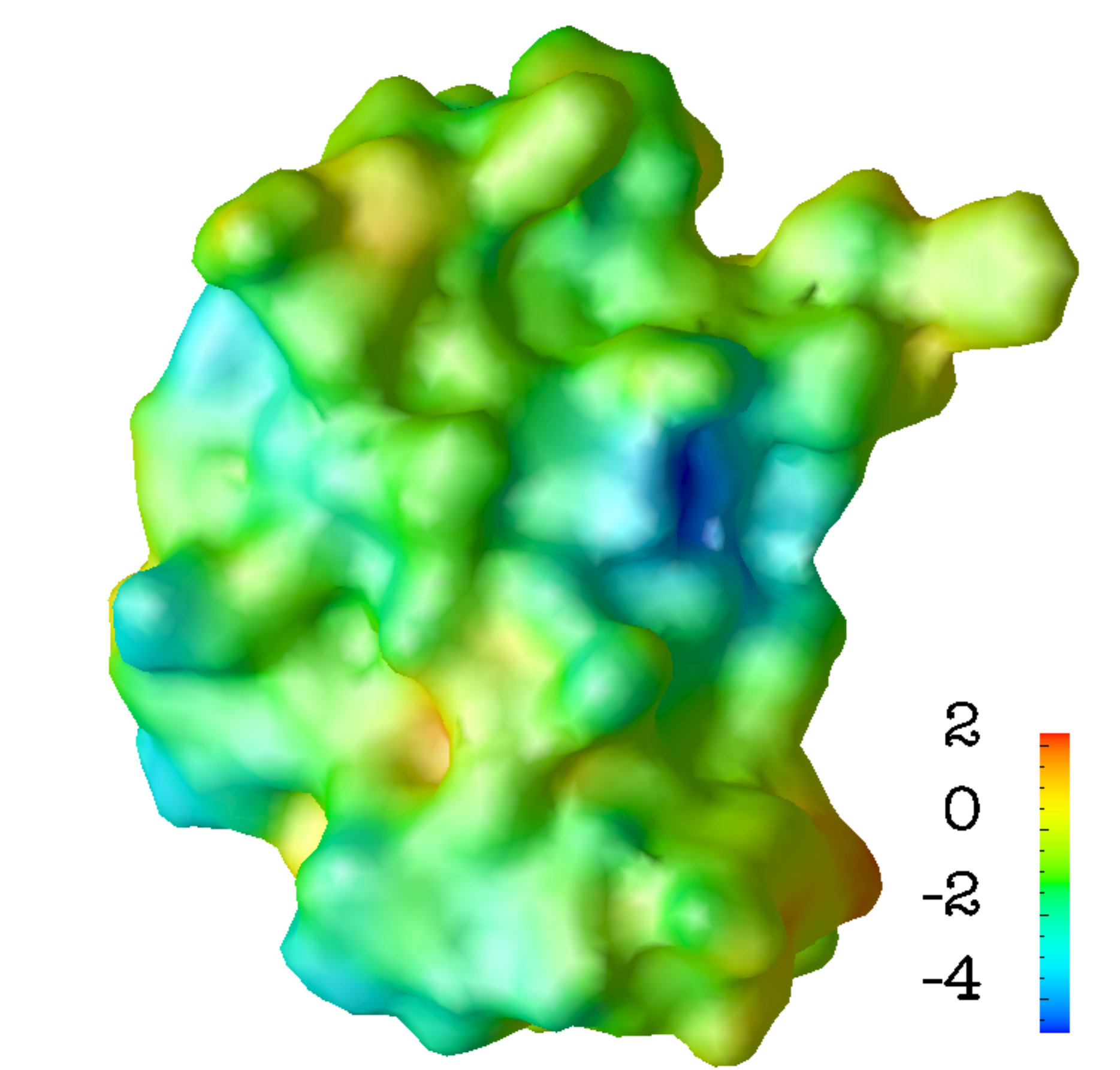}
   \includegraphics[width=0.45\textwidth,angle=90]{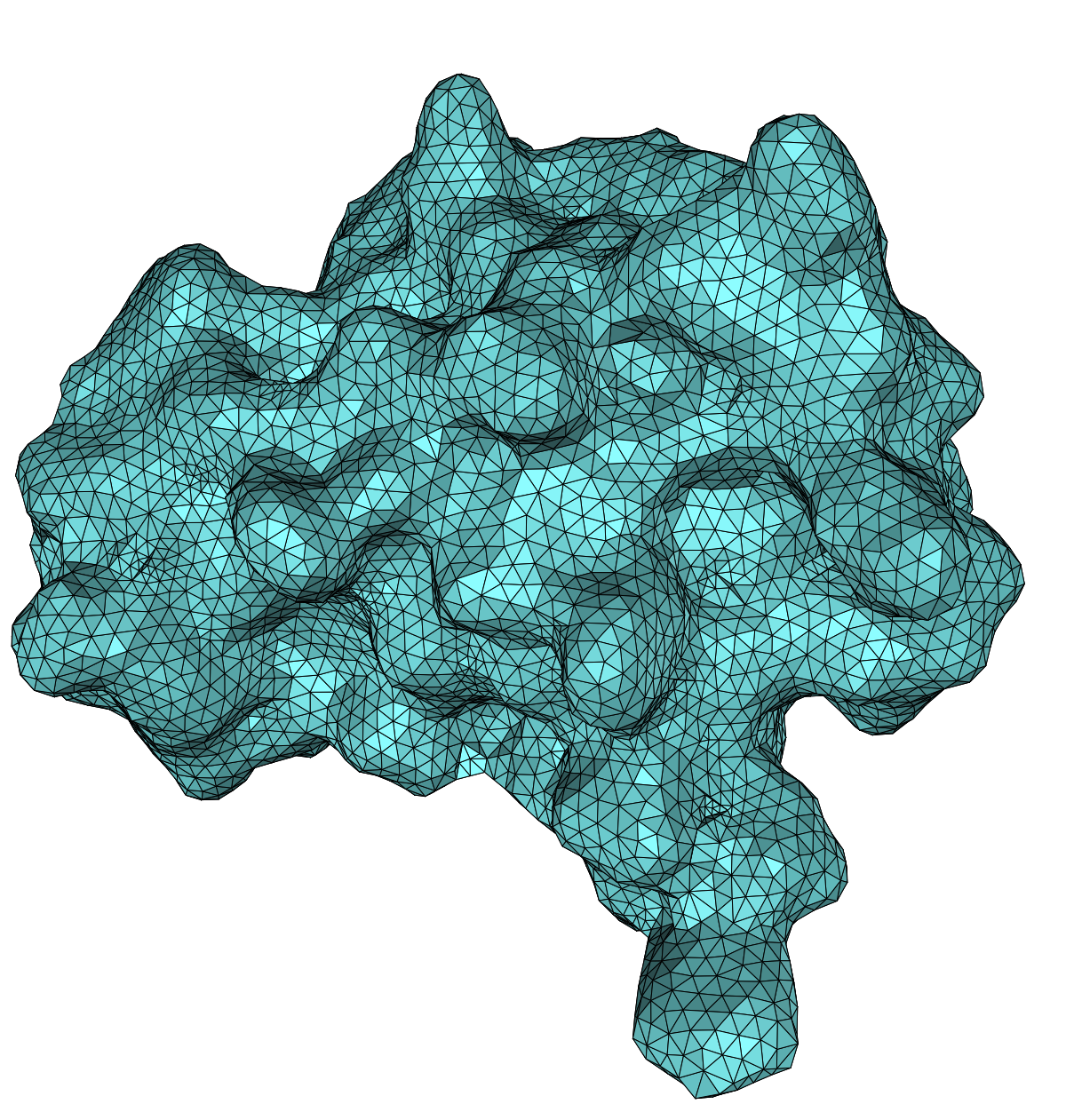} 
\end{center}
\caption{Left: The electrostatic potential mapped on the molecular surface 
         of the insulin protein.
         Right: The surface mesh of insulin protein in the finite element model.}
\label{fig:protein_pot_mesh}
\end{figure}

\begin{table}[!ht]
\begin{center}
\begin{tabular}{cccccc} \hline
 level of mesh &  \# of tetrahedra  & \# of nodes  & \# of nodes on $\Gamma$  & $\Delta G_{ele}$  & $e_{\Delta G_{ele}}$ \\ \hline
   1 & 280928 & 45119 & 7681 & -1476.9 & 0.0757 \\  
   2 & 340410 & 54685 & 10257 & -1403.6 &  0.0224 \\ 
   3 & 474011 & 76036 & 14982 & -1380.3 & 0.0054\\
   4 & 626221 & 100393 & 19816 & -1372.9 & - \\ \hline
\end{tabular}
\end{center}
\caption{
The electrostatic solvation energy $\Delta G_{ele}$ and
the corresponding relative error $e_{\Delta G_{ele}}$ for 
progressively refined meshes, confirming convergence of the
discretization technique based on the new regularization.
}\label{table:accuracy}
\end{table}

%% file: conc.tex
%%%%%%%%%%%%%%%%%%%%%%%%%%%%%%%%%%%%%%%%%%%%%%%%%%%%%%%%%%%%%%%%%%%%%%%%%%%%%%
\section{Summary}
\label{sec:summary}

In this article, we considered the design of an effective and reliable 
adaptive finite element
method (AFEM) for the nonlinear Poisson-Boltzmann equation (PBE).
In Section~\ref{sec:pbe}, we began with a very brief derivation of
the standard form of the Poisson-Boltzmann equation.
We examined the two-scale regularization technique 
described in~\cite{CHX06b}, and briefly reviewed the
solution theory ({\em a priori} estimates and other basic results)
developed in~\cite{CHX06b} based on this regularization.
We then described a second distinct regularization and explained why it is
superior to the original approach as a framework for developing numerical 
methods.
We then quickly assembled the cast of basic mathematical results needed
for the second regularization.
In Section~\ref{sec:afem}, we described in detail an adaptive finite
element method based on residual-type {\em a posteriori} estimates, and
summarized some basic results we needed later for the development of a
corresponding convergence theory.
We presented this new convergence analysis in Section~\ref{sec:conv}, giving
the first AFEM contraction-type result for a class of semilinear problems
that includes the Poisson-Boltzmann equation.

We gave a detailed discussion of our mesh generation toolchain
in Section~\ref{sec:mesh}, including
algorithms designed specifically for Poisson-Boltzmann applications.
These algorithms produce a high-quality, high-resolution geometric model 
(surface and volume meshes) satisfying the assumptions needed for our 
AFEM algorithm.
These algorithms are {\em feature-preserving and adaptive},
designed specifically for constructing meshes of biomolecular 
structures, based on the intrinsic local structure tensor of the 
molecular surface.
Numerical experiments were given in Section~\ref{sec:numerical};
all of the AFEM and meshing algorithms described in the article were 
implemented in the Finite Element Toolkit (FETK), developed and 
maintained at UCSD.
The stability advantages of the new regularization scheme were demonstrated 
with FETK through comparisons with the original regularization approach
for a model problem.
Convergence and accuracy of the AFEM algorithm was also illustrated 
numerically by approximating the solvation energy for a protein,
in agreement with theoretical results established earlier in the paper.

In this article, we have examined an alternative regularization
which must be used in place of the original regularization
proposed in~\cite{CHX06b}, due to an
inherent instability built into the original regularization.
We showed that an analogous solution and approximation theory 
framework can be put into in place for the new regularization, providing 
a firm foundation for the development of a large class of numerical methods 
for the Poisson-Boltzmann equation, including methods based on finite 
difference, finite volume, spectral, wavelet and finite element methods.
Each of these methods can be shown to be convergent for the regularized
problem, since it was shown in this article to allow for a standard
$H^1$ weak formulation with standard solution and test spaces.
Our primary focus in this article then became the development of an 
AFEM scheme for the new regularized problem, based on
residual-type {\em a posteriori} error indicators, a fairly standard
and easy to implement marking strategy (D\"orfler marking), 
and well-understood simplex bisection algorithms.
We showed that the resulting AFEM scheme is reliable, by proving
a contraction result for the error, which established convergence
of the AFEM algorithm to the solution of the continuous problem.
The AFEM contraction result, which is one of the first results of this 
type for nonlinear elliptic problems, follows from the global upper boundedness 
of the estimator, its reduction, and from a quasi-orthogonality result
that relies on the {\em a priori} $L^{\infty}$ estimates we derived.
This new AFEM convergence framework is distinct from the analysis of 
nonlinear PBE with the previous regularization approach from 
\cite{CHX06b}, is more general, and can be applied 
to other semi-linear elliptic equations~\cite{HTZ08a}.
The contraction result creates the possibility of establishing
optimality of the AFEM algorithm in both computational and storage complexity.

We note that for computational chemists and physicists who rely on numerical 
solution of the Poisson-Boltzmann equation, discretizations based on the 
stable splitting as described in the current paper are the only reliable 
numerical techniques under mesh refinement for the Poisson-Boltzmann equation
that we are aware of (both provably convergent and stable to roundoff error).
While one must take care with evaluation of the singular function $u^s$, 
since this generally involves pairwise interactions between charges and mesh
points, the alternative to using these types of splitting discretizations
is to lose reliability in the quality of the numerical solution.
While we focused on (adaptive) finite element methods in this article, 
we emphasize that the splitting framework can be easily incorporated 
into one's favored (finite difference, finite volume, spectral, wavelet, 
or finite element) numerical method that is currently being employed
for the PBE.

%% file: ack.tex
\section*{Acknowledgments}

% NSF DMS/CM Award 0915220 
% NSF DMS/CM Award 0715146 
% NSF DMS/MRI Award 0821816 
% NSF PHY/PFC Award 0822283 
% NIH P41 Award P41RR08605-16 
% DOD/DTRA Award HDTRA-09-1-0036 
% UC Lab Research Program Award 0118418 

We would like to thank the reviewer for comments that 
improved the content and readability of the paper.
The first author was supported in part by 
NSF Awards~0715146, 0821816, 0915220, and 0822283 (CTBP),
NIH Award P41RR08605-16 (NBCR), DOD/DTRA Award HDTRA-09-1-0036,
CTBP, NBCR, NSF, and NIH.
The second author was supported in part by NIH, NSF, HHMI, CTBP, and NBCR.
The third, fourth, and fifth authors
were supported in part by NSF Award~0715146, CTBP, NBCR, and HHMI.